\documentclass[smallcondensed]{svjour3}      
\smartqed  
\usepackage{graphicx}

\usepackage{amssymb,amsmath, color}
\usepackage{amssymb}
\usepackage{amsbsy}
\usepackage{enumitem,algorithm2e,algorithmic}

\usepackage{ulem}

\usepackage[colorlinks=true]{hyperref}

\newcommand{\R}{{\mathbb R}}

\DeclareMathOperator{\argmin}{argmin}
\DeclareMathOperator{\prox}{prox}

\newcommand{\cH}{{\mathcal H}}

\newcommand{\demi}{\frac{1}{2}}
\newcommand{\ie}{{\it i.e.}\,\,}

\newlength{\textlarg} 

\newcommand{\interior}{{\rm int}\kern 0.06em}
\newcommand{\inte}{{\rm int}\kern 0.06em}
\newcommand{\cl}{{\rm cl}\kern 0.06em}
\newcommand{\zer}{{\rm zer}\kern 0.06em}
\newcommand{\gph}{{\rm gph}\kern 0.06em}
\newcommand{\dom}{{\rm dom}\kern 0.06em}
\newcommand{\pr}{{\rm pr}\kern 0.06em}
\newcommand{\e}{\varepsilon}

\def\d{\delta}

\def\<{\langle}
\def\>{\rangle}

\newcommand{\tcb}{\textcolor{blue}}
\newcommand{\tcr}{\textcolor{red}}

\usepackage{epsfig}
\usepackage{geometry}
 \usepackage{pict2e}

\if
 {
 \usepackage[pageref]{backref}
\renewcommand*{\backrefalt}[4]{%
\ifcase #1 %
(Not cited)%
\or
(Cited on p.~#2)%
\else
(Cited on pp.~#2)%
\fi
}

}
\fi

\usepackage{ulem}

\begin{document}

\title{Damped inertial dynamics with vanishing Tikhonov regularization: strong asymptotic convergence towards the minimum norm solution}

\titlerunning{Inertial gradient dynamics with Tikhonov regularization}

\author{Hedy ATTOUCH \and A\"icha BALHAG   \and Zaki CHBANI   \and Hassan RIAHI}

\institute{
Hedy ATTOUCH  \at IMAG, Univ. Montpellier, CNRS, Montpellier, France\\
hedy.attouch@umontpellier.fr,\\  Supported by COST Action: CA16228
\and
A\"icha BALHAG   \and Zaki CHBANI   \and Hassan RIAHI\\
 Cadi Ayyad University \\ S\'emlalia Faculty of Sciences 
 40000 Marrakech, Morroco\\
 aichabalhag@gmail.com \and   chbaniz@uca.ac.ma  \and h-riahi@uca.ac.ma 
}
\maketitle


\begin{abstract}
  In a Hilbert space, we provide a fast dynamic approach to the hierarchical minimization problem which consists in finding the minimum norm solution of a convex minimization problem.
For this,  we study the convergence properties of the trajectories generated by a damped inertial dynamic with Tikhonov regularization. When the time goes to infinity, the Tikhonov regularization parameter is supposed to tend towards zero, not too fast, which is a key property to make the  trajectories  strongly converge  towards the minimizer of $f$ of minimum norm.
According to the structure of the heavy ball method for strongly convex functions, the viscous  damping coefficient  is proportional to the square root of the Tikhonov regularization parameter. Therefore, it also converges to zero, which will ensure rapid convergence of  values.
Precisely, under a proper tuning of these parameters, based on Lyapunov's analysis, we show that the trajectories  strongly converge  towards the minimizer  of minimum norm, and we provide the convergence rate of the values. We show a trade off between the property of fast convergence of values, and the property of strong convergence towards the minimum norm solution.
This study improves several previous works where this type of results was obtained under restrictive hypotheses.  
\end{abstract}

\medskip


\medskip



\vspace{5mm}

\section{Introduction}
Throughout the paper, $\mathcal H$ is a real Hilbert space which is endowed with the scalar product $\langle \cdot,\cdot\rangle$, with $\|x\|^2= \langle x,x\rangle    $ for  $x\in \mathcal H$.
 We consider
 the convex minimization problem
\begin{equation}\label{edo0001}
 \min \left\lbrace  f (x) : \ x \in \mathcal H \right\rbrace,
\end{equation}
where $f : \mathcal H \rightarrow \mathbb R$ is a convex continuously differentiable function whose solution set $S=\argmin f$ is  nonempty.
We aim at finding  by rapid methods the element of minimum norm of $S$.
Our approach is in line with the dynamic approach
 developed  by Attouch and L\'aszl\'o in \cite{AL} to solve this question. It is based on the asymptotic analysis, as $t \to +\infty$, of the nonautonomous damped inertial dynamic
 \begin{equation*}
{\rm(TRIGS)} \qquad \ddot{x}(t) + \d\sqrt{\e(t)}  \dot{x}(t) + \nabla f (x(t)) + \varepsilon (t) x(t) =0,
\end{equation*}
where the function  $f$ and the Tikhonov regularization parameter $\varepsilon$ satisfy the following hypothesis\footnote{In section \ref{non-smooth}, we will extend our study to the case of a convex lower semicontinuous proper function $f: \cH \to \R \cup \left\lbrace +\infty \right\rbrace$.}:
\begin{align*}
( \mathcal{H}_0) \;\begin{cases}
 \; \; f : \mathcal H \rightarrow \mathbb R \mbox{ is convex and differentiable},  \nabla f \mbox{ is Lipschitz continuous on  bounded sets}; \vspace{1mm} \\
 \; \; S := \mbox{argmin} f \neq \emptyset. \mbox{ We denote by } x^*  \mbox{ the element of minimum norm of } S;   \vspace{1mm}\\
\; \;  \varepsilon : [t_0 , +\infty [ \to \mathbb R^+  \mbox{ is   a nonincreasing function, of class } \mathcal C^1, \mbox{ such that }\  \lim_{t \to \infty} \varepsilon (t) =0.
\end{cases}
\end{align*}
The Cauchy problem for (TRIGS) is well posed.
The proof of the existence and uniqueness of a global solution for the corresponding Cauchy problem is given in the appendix (see also \cite{AL}). It  is based on classical arguments  combining the Cauchy-Lipschitz theorem with energy estimates. 
Our main contribution is to develop a new Lyapunov analysis which gives the strong convergence of the trajectories of (TRIGS) to the element of minimum norm of $S$. Precisely, we give sufficient conditions on $\varepsilon (t)$ which ensure that $\lim_{t\to +\infty}\| x(t)-x^\ast\| =0$. This improves the results of \cite{AL}.

\subsection{Attouch-L\'aszl\'o Lyapunov analysis of (TRIGS)}

The main idea  developed in \cite{AL} consists of starting from the Polyak heavy ball with friction dynamic for strongly convex functions,  then adapting it via Tikhonov approximation to deal with the case of general convex functions.
 Recall that
a function  $f: \cH \to \mathbb R$ is said to be $\mu$-strongly convex for some $\mu >0$ if   $f- \frac{\mu}{2}\| \cdot\|^2$ is convex.
In this setting, we have the following exponential convergence result for the damped autonomous inertial dynamic where the damping coefficient is twice the square root of the modulus of strong convexity of $f$:
\begin{theorem}\label{strong-conv-thm}
Suppose that $f: \cH \to \mathbb R$ is a function of class ${\mathcal C}^1$ which is $\mu$-strongly convex for some $\mu >0$.
Let  $x(\cdot): [t_0, + \infty[ \to \cH$ be a solution trajectory of
\begin{equation}\label{dyn-sc-a}
\ddot{x}(t) + 2\sqrt{\mu} \dot{x}(t)  + \nabla f (x(t)) = 0.
\end{equation}
 Then, the following property holds:  
$$
f(x(t))-  \min_{\mathcal H}f  = \mathcal O \left( e^{-\sqrt{\mu}t}\right) \;  \mbox{ as } \; t \to +\infty.
$$
\end{theorem}
\noindent To adapt this result to the case of a general convex differentiable function $f: \cH \to \mathbb R$, a natural idea  is to use Tikhonov's method of regularization. This leads to consider the  non-autonomous dynamic which at time $t$ is governed by the gradient of the strongly convex function 
 $$\varphi_t: \cH \to \mathbb R, \quad
\varphi_t (x) := f(x) + \frac{\varepsilon(t)}{2} \|x\|^2.
 $$
  Then, replacing $f$ by $\varphi_t$ in
\eqref{dyn-sc-a}, and noticing that $\varphi_t$ is $\varepsilon (t)$-strongly convex, this gives the following dynamic which was introduced in \cite{AL} ($\delta$ is a positive parameter)
\begin{equation*}
{\rm(TRIGS)} \qquad \ddot{x}(t) + \d\sqrt{\e(t)}  \dot{x}(t) + \nabla f (x(t)) + \varepsilon (t) x(t) =0.
\end{equation*}
 (TRIGS) stands shortly for Tikhonov regularization of inertial gradient systems.
In order not to asymptotically modify the equilibria, it is supposed that $\varepsilon (t) \to 0$ as $t\to +\infty$\footnote{This is the key property of the asymptotic version ($t\to +\infty$) of the Browder-Tikhonov regularization method.}. This condition implies that (TRIGS) falls within the framework of the inertial gradient systems with asymptotically vanishing damping.
The importance of this  class of inertial dynamics  has been highlighted by several recent studies
\cite{AAD1}, \cite{ABotCest}, \cite{AC10}, \cite{ACPR}, \cite{AP}, \cite{CD}, \cite{SBC}, which make the link with the accelerated gradient method of Nesterov  \cite{Nest1,Nest2}.\\
The control of the decay of $\varepsilon(t)$ to zero as $t \to +\infty$  plays a key role in the  Lyapunov analysis of (TRIGS), and uses  the following condition.
\begin{definition}
Let us give $\delta >0$. We say that $t \mapsto \varepsilon(t)$ satisfies the controlled decay property ${\rm(CD)}_{\lambda}$, if it is a nonincreasing function which satisfies: there exists $t_1\ge t_0$ such that for all $t\ge t_1,$
\begin{equation}\label{AL-growth}
\dfrac{d}{dt} \left(\frac{1}{\sqrt{\e(t)}}\right) \le \min(2\lambda-\d, \d-\lambda),
\end{equation} 
where $\lambda$ is a  parameter such that $ \frac{\d}{2} < \lambda  < \delta$ for $0<\delta \leq 2$,  and $ \frac{\d+ \sqrt{\delta^2 -4}}{2} < \lambda  < \delta$ for $\delta > 2$ .
\end{definition}

By integrating the differential inequality \eqref{AL-growth}, one can easily verify that
this condition implies that
   $ \e (t) $ is greater than or equal to $ C / t^2 $.  Since the damping coefficient is proportional to $\sqrt{\e(t)}$, this means that it must be greater than or equal to $C/t$.
 This is in accordance with the theory of inertial gradient systems with time-dependent viscosity coefficient, which states that the asymptotic optimization property is valid provided that the integral on $[t_0, +\infty[$ of the viscous damping coefficient  is infinite,  see \cite{AC10}, \cite{CEG}.
 Let us state the following convergence result obtained in \cite{AL}.

\begin{theorem}\label{AL1}{\rm (Attouch-L\'aszl\'o \cite{AL})}
Let $x : [t_0, +\infty[ \to \mathcal{H}$ be a solution trajectory of  {\rm(TRIGS)}. Let $\delta$ be a positive parameter.
Suppose that $\varepsilon(\cdot)$ satisfies the condition ${\rm(CD)}_{\lambda}$ for some $\lambda>0$.
Then,  we have the following rate of convergence of values: for all $t\ge t_1$
\begin{equation}\label{basic-Lyap}
f(x(t)) - \min_{\cH} f \le \frac{\lambda\|x^*\|^2}{2}\frac{1}{\gamma(t)}\int_{t_1}^t \e^{\frac32}(s)\gamma(s)ds+\frac{C}{\gamma(t)},
\end{equation}
where
\begin{equation*}
\gamma(t)=\exp \left({\displaystyle{\int_{t_1}^t\mu(s)ds}}\right), \quad
\mu(t) =-\frac{\dot{\varepsilon}(t)}{2\varepsilon(t)}+ (\delta-\lambda)\sqrt{\varepsilon(t)}
\end{equation*}
\begin{flushleft}
$ \mbox{ and  }\; C=\left( f(x(t_1)) - f^\ast \right) +\frac{\e(t_1)}{2}\|x(t_1)\|^2+ \frac{1}{2} \| \lambda\sqrt{\e(t_1)}(x(t_1)-x^\ast) + \dot{x}(t_1) \|^2.$
\end{flushleft}
\end{theorem}
The proof is based on the following Lyapunov function $\mathcal{E} : [t_0, +\infty[ \to \mathbb{R}_+,$
\begin{align}\label{Lyapunov-a}
\mathcal{E}(t):= \left( f(x(t)) - \min f \right) +\frac{\e(t)}{2}\|x(t)\|^2+ \frac{1}{2} \| c(t)(x(t)-x^\ast) + \dot{x}(t) \|^2,
\end{align}
where
the function $c:[t_0,+\infty[\to\R$ is chosen appropriately.
Based on this Lyapunov analysis, it is proved in \cite{AL} that $\liminf_{t\to+\infty} \| x(t)-x^\ast\| =0$.
We will improve this result, and show that $\lim_{t\to+\infty}
\| x(t)-x^\ast\| =0$. For this, we will develop a new Lyapunov analysis. Let us first recall some related previous results showing the progression of the understanding of these delicate questions, where the hierarchical minimization property is reached asymptotically. 

\subsection{Historical facts and related results}

In relation to hierarchical optimization, a rich literature has been devoted to the coupling of dynamic gradient systems with Tikhonov regularization, and to the study of the corresponding algorithms.

\subsubsection{First-order gradient dynamics}
 For first-order gradient systems and subdifferential inclusions, the asymptotic hierarchical minimization property which results from the introduction of a vanishing viscosity term in the dynamic (in our context the Tikhonov approximation \cite{Tikh,TA}) has been highlighted in a series of papers  \cite{AlvCab}, \cite{Att2},   \cite{AttCom}, \cite{AttCza2}, \cite{BaiCom}, \cite{CPS}, \cite{Hirstoaga}.
In parallel way, there is  a vast literature on convex descent algorithms involving Tikhonov and more
general penalty, regularization terms. The  historical evolution
 can be traced back to Fiacco and McCormick \cite{FM}, and the interpretation of interior point methods with the help of a vanishing logarithmic barrier.
Some more specific references for the coupling of Prox and Tikhonov can be found in
Cominetti \cite{Com}.
The time discretization  of  the first-order gradient systems and subdifferential inclusions involving multiscale (in time) features provides a natural link between the continuous and discrete  dynamics. The  resulting algorithms  combine  proximal based methods (for example forward-backward algorithms), with the viscosity of penalization methods, see
\cite{AttCzaPey1}, \cite{AttCzaPey2}, \cite{BotCse1}, \cite{Cabot-inertiel,Cab}, \cite{Hirstoaga}.

\medskip

\subsubsection{Second order gradient dynamics}

First studies concerning the coupling of damped inertial dynamics with Tikhonov approximation concerned  the heavy ball with friction system of Polyak \cite{Polyak},
where the damping coefficient $\gamma >0$ is  fixed. In   \cite{AttCza1} Attouch-Czarnecki considered the  system
\begin{equation}\label{HBF-Tikh}
 \ddot{x}(t) + \gamma \dot{x}(t) + \nabla f(x(t)) + \varepsilon (t) x(t) =0.
\end{equation}
In the slow parametrization case $\int_0^{+\infty} \varepsilon (t) dt = + \infty$, they proved that  any solution $x(\cdot)$ of \eqref{HBF-Tikh} converges strongly to the minimum norm element of $\argmin f$, see also \cite{JM-Tikh}.
 A parallel study has been developed  for PDE's, see \cite{AA} for damped hyperbolic equations with non-isolated equilibria, and
\cite{AlvCab} for semilinear PDE's.
The system \eqref{HBF-Tikh} is a special case of the general dynamic model
\begin{equation}\label{HBF-multiscale}
\ddot{x}(t) + \gamma \dot{x}(t) + \nabla f (x(t)) + \varepsilon (t) \nabla g (x(t))=0
\end{equation}
which involves two  functions $f$ and $g$ intervening with different time scale. When $\varepsilon (\cdot)$ tends to zero moderately slowly, it was shown in \cite{Att-Czar-last} that the trajectories of (\ref{HBF-multiscale})  converge asymptotically to equilibria that are solutions of the following hierarchical problem: they minimize the function $g$ on the set of minimizers of $f$.
When $\mathcal H= {\mathcal H}_1\times {\mathcal H}_2$ is a product space, defining for $x=(x_1,x_2)$, $f (x_1,x_2):= f_1 (x_1)+f_2 (x_2)$ and   $g(x_1,x_2):= \|A_1 x_1 -A_2 x_2 \|^2$, where the $A_i,\, i\in\{1,2\}$ are linear operators, (\ref{HBF-multiscale})  provides (weakly) coupled inertial systems.
The continuous and discrete-time versions of these systems have a natural connection to the best response dynamics for potential games \cite{AttCza2}, domain decomposition for PDE's \cite{abc2}, optimal transport \cite{abc}, coupled wave equations \cite{HJ2}.

\noindent In the quest for a faster convergence, the following system
\begin{equation}\label{edo001-0}
 \mbox{(AVD)}_{\alpha, \varepsilon} \quad \quad \ddot{x}(t) + \frac{\alpha}{t} \dot{x}(t) + \nabla f (x(t)) +\varepsilon(t) x(t)=0,
\end{equation}
has been studied by Attouch-Chbani-Riahi \cite{ACR}.
It is a Tikhonov regularization of the  dynamic
\begin{equation}\label{edo001}
 \mbox{(AVD)}_{\alpha} \quad \quad \ddot{x}(t) + \frac{\alpha}{t} \dot{x}(t) + \nabla f (x(t))=0,
\end{equation}
which was introduced by  Su, Boyd and
Cand\`es in \cite{SBC}. When $\alpha =3$, $\mbox{(AVD)}_{\alpha}$  can be viewed as a continuous version of the   accelerated gradient method of Nesterov.
It has been the subject of many recent studies which have given an in-depth understanding of the Nesterov acceleration method, see  \cite{AAD1}, \cite{AC10}, \cite{ACPR}, \cite{SBC}, \cite{Siegel}, \cite{WRJ}.

\subsection{Model results}
Let us illustrate our results in the case $\varepsilon (t) = \frac{1}{t^p}$. In section \ref{sec: particular}, we will prove the following result:

\begin{theorem}\label{thm:model-a}
Take $\e(t)=1/t^p$,\;   $0<p<2$.
Let $x : [t_0, +\infty[ \to \mathcal{H}$ be a solution trajectory of
$$
\ddot{x}(t) +  \frac{\d}{t^{\frac{p}{2}}}\dot{x}(t) + \nabla f\left(x(t) \right)+ \frac{1}{t^p} x(t)=0.
$$
Then, we have  the following rates of convergence for  the values and the trajectory:
\begin{eqnarray}
&&f(x(t))-\min_{\cH} f= \mathcal O \left( \displaystyle{ \frac{1}{t^{p} }  } \right) \mbox{ as } \; t \to +\infty\\
&&\|x(t) - x_{\varepsilon(t)}\|^2=\mathcal{O}\left(\dfrac{1}{ t^{\frac{2-p}2}}\right) \mbox{ as } \; t \to +\infty.
\end{eqnarray}
According to the strong  convergence of $x_{\varepsilon(t)}$ to the minimum norm solution, the above property  implies that 
$x(t)$ strongly converges to the  minimum norm solution.
\end{theorem}

With many respect, these results represent an important advance compared to previous works:

\smallskip

i) Let us underline  that in our approach the rapid convergence of the values and the strong convergence towards the solution of  minimum norm are obtained for the same dynamic, whereas in the previous works \cite{ACR}, \cite{AttCza1}, they are obtained for different dynamics corresponding to different settings of the parameters.
Moreover, we obtain the strong convergence of the trajectories to 
the minimum norm solution, whereas in \cite{AL} and \cite{ACR} it was only obtained that $\liminf_{t \to + \infty}{\| x(t) - x^\ast \|} = 0.$
It is clear that the results extend naturally to obtaining  strong convergence towards the solution closest to a desired state $x_d$.
It suffices to replace in Tikhonov's approximation $\|x\|^2$ by 
$\|x-x_d\|^2$. This is important for inverse problems.
In addition, we obtain a convergence rate of the values which is
better than the one obtained in\cite{AL}. 

\smallskip

ii) These results show the trade-off between the property of rapid convergence of values, and the property of strong convergence towards the minimum norm solution. The two rates of convergence
move in opposite directions with $p$ varies.
The determination of a good compromise between these two antagonistic criteria is an interesting subject that we will consider later.

\smallskip

iii) Note that at the limit, when $ p = 2 $, which is the most interesting case to obtain a fast convergence of values comparable to the accelerated gradient method of Nesterov, then our analysis does not allow us to conclude that the trajectories are converging towards
the solution of  minimum norm. This question remains open,
the interested reader can consult \cite{AL}.

\subsection{Contents}
The paper is organized as follows.
In section \ref{sec:general}, for a general Tikhonov regularization parameter $\varepsilon(t)$, we study the asymptotic convergence properties of the solution trajectories of (TRIGS). Based on Lyapunov analysis, we show their strong convergence to the element of minimum norm of $S$, and we provide  convergence rate of the values.
In section \ref{sec: particular}, we apply these results to 
the particular case $\varepsilon (t) =\frac{1}{t^p}$.
Section \ref{non-smooth} considers the extension of these results to the nonsmooth case.
Section \ref{num} contains numerical illustrations. We conclude in section \ref{sec:Conclusion}
with some perspective and open questions.

\section{Convergence analysis for general $\varepsilon(t)$}\label{sec:general}
We are going to analyze via Lyapunov analysis the convergence properties as $t\to +\infty$ of the solution trajectories of the inertial dynamic (TRIGS) that we recall below
\begin{equation}\label{1}
 \ddot{x}(t) + \d\sqrt{\e(t)}  \dot{x}(t) + \nabla f (x(t)) + \varepsilon (t) x(t) =0.
\end{equation}
Throughout the paper, we assume that $t_0$ is the origin of time, $\delta$ is a positive parameter.
For each $t \geq t_0$, let us introduce the function $\varphi_{t} : \cH \to \R$ defined by
\begin{equation}\label{phi}
\varphi_{t}(x):= f(x)+\dfrac{\varepsilon(t)}{2}\|x\|^{2},
\end{equation}
and set
$$
x_{\varepsilon(t)}:= \argmin_{\cH}\varphi_{t},
$$
which is the unique minimizer of the strongly convex function $\varphi_{t}$. The first order optimality condition gives 
\begin{equation}\label{opt}
\nabla f(x_{\varepsilon(t)})+\varepsilon(t)x_{\varepsilon(t)}=0.
\end{equation}
The following properties are immediate consequences of  the classical properties of the Tikhonov regularization
\begin{eqnarray}
&& \forall t\geq t_0 \;\;  \|x_{\varepsilon(t)}\|\leq \|x^{*}\|
\label{2a}  \vspace{3mm}\\
&&  \lim_{t\rightarrow +\infty}\|x_{\varepsilon(t)}-x^{*}\|=0 \quad\hbox{where}\: x^{*}=\mbox{proj}_{\argmin f} 0.\label{2b}
\end{eqnarray}

\subsection{Preparatory results for Lyapunov analysis}
Let us introduce the real-valued function function $t \in [t_0, +\infty[ \mapsto E(t) \in \R^+$ that plays a key role in our Lyapunov analysis. It is defined by
\begin{equation}\label{3}
E(t):=\left(\varphi_{t}(x(t))-\varphi_{t}(x_{\varepsilon(t)})\right) +\frac{1}{2}\|v(t)\|^{2},
\end{equation}
where $\varphi_{t}$ has been defined in (\ref{phi}), and
\begin{equation}\label{3b}
v(t)=\tau(t)\left(x(t)-x_{\varepsilon(t)}\right)+\dot{x}(t).
\end{equation}
The time dependent parameter $\tau (\cdot)$  will be adjusted during the proof.
We will show that under judicious choice of the parameters, then $t \mapsto E(t)$ is a decreasing function. Moreover, we will estimate the rate of convergence of $E(t)$ towards zero. 
This will provide the rates of convergence of values and trajectories, as the following lemma shows.

\begin{lemma}\label{lem-basic}
Let   $x(\cdot): [t_0, + \infty[ \to \cH$ be a solution trajectory of  the damped inertial dynamic {\rm(TRIGS)}, and $t \in [t_0, +\infty[ \mapsto E(t) \in \R^+$ be the energy function defined in \eqref{3}. Then, the following inequalites are satisfied:   for any $t\geq t_0$, 

\begin{eqnarray}
&&f(x(t))-\min_{\mathcal H}f	
	 \leq  E(t)+\dfrac{\varepsilon(t)}{2}\|x^*\|^{2}; \label{keybb}\\
	&&\|x(t) - x_{\varepsilon(t)}\|^2  \leq \frac{2E(t)}{\varepsilon(t)} \label{est:basic1}.
\end{eqnarray}
 Therefore, $x(t)$ converges strongly to $x^*$
as soon as 
$
\lim_{t\to +\infty} \displaystyle{\frac{E(t)}{\varepsilon(t)}}=0.
$
 \end{lemma}
 \begin{proof}
 
 $i)$
 According to the definition of $ \varphi_{t}$,  we have 
 \begin{equation*}
\begin{array}{lll}
f(x(t))-\min_{\mathcal H}f	& = &  \varphi_{t}(x(t))-\varphi_{t}(x^*)+\dfrac{\varepsilon(t)}{2}\left(\|x^*\|^{2}-\|x(t)\|^{2}\right)  \\ 
	& = & \left[\varphi_{t}(x(t))-\varphi_{t}(x_{\varepsilon(t)})\right]+\left[\underbrace{\varphi_{t}(x_{\varepsilon(t)})-\varphi_{t}(x^*)}_{\leq 0}\right]+\dfrac{\varepsilon(t)}{2}\left(\|x^*\|^{2}-\|x(t)\|^{2}\right)\\
	& \leq  &\varphi_{t}(x(t))-\varphi_{t}(x_{\varepsilon(t)})+\dfrac{\varepsilon(t)}{2}\|x^*\|^{2}.
\end{array}
\end{equation*}
By definition of $E(t)$ we have 
 \begin{equation}\label{E_phi}
  \varphi_{t}(x(t))-\varphi_{t}(x_{\varepsilon(t)}) \leq E(t)
  \end{equation}
 which, combined with the above inequality, gives \eqref{keybb}.

\smallskip

$ii)$
By the strong convexity of $\varphi_{t}$, and 
$x_{\varepsilon(t)}:= \argmin_{\cH}\varphi_{t}$, we have
$$
\varphi_{t}(x(t))-\varphi_{t}(x_{\varepsilon(t)} \geq \frac{\varepsilon (t)}{2}  \|x(t) - x_{\varepsilon(t)}\|^2 .
$$
By combining the  inequality above with \eqref{E_phi}, we get
$$
E(t) \geq \frac{\varepsilon (t)}{2}  \|x(t) - x_{\varepsilon(t)}\|^2 ,
$$
which gives \eqref{est:basic1}.\\
 By the assumption $ \mathcal{H}_0)$, we have $\lim_{t \to \infty} \varepsilon (t) =0$. According to \eqref{2b}, we deduce that  $x_{\varepsilon(t)}$ converges strongly to $x^*$.
 The conclusion is a direct consequence of inequality \eqref{est:basic1}.
\end{proof}

To estimate $E(t)$, we will show that it satisfies a first order differential inequality of the form
\begin{equation}\label{est:basic2}
 \dot{E}(t)+\mu(t)E(t)\leq \rho (t) \|x^{*}\|^{2}
 \end{equation}
\noindent  where $\rho(t)$ and $\mu(t)$ are positive functions that will be made precise in the proof.
 So the first step of the proof is to compute $\dfrac{d}{dt}E(t)$.
The computation  of $\dfrac{d}{dt}E(t)$  involves the two terms$\dfrac{d}{dt}\left(\varphi_{t}(x_{\varepsilon(t)})\right)$ and $\dfrac{d}{dt}\left(x_{\varepsilon(t)}\right).$ They are evaluated in the lemma below.
\begin{lemma}\label{lem1}
For each $t \geq t_0$, we have
 \begin{itemize}
 	\item[$i)$] $\dfrac{d}{dt}\left(\varphi_{t}(x_{\varepsilon(t)})\right)=\frac{1}{2}\dot{\varepsilon}(t)\|x_{\varepsilon(t)}\|^{2};$
 	\item[$ii)$] $\left\|\dfrac{d}{dt}\left(x_{\varepsilon(t)}\right)\right\|^{2} \leq -\dfrac{\dot{\varepsilon}(t)}{\varepsilon(t)} \left\langle \dfrac{d}{dt}\left(x_{\varepsilon(t)}\right), x_{\varepsilon(t)}\right\rangle$.
 \end{itemize}
 	Therefore, $\left\|\dfrac{d}{dt}\left(x_{\varepsilon(t)}\right)\right\|\leq -\dfrac{\dot{\varepsilon}(t)}{\varepsilon(t)} \|x_{\varepsilon(t)}\|. $
 
\end{lemma}
\begin{proof}  
	$i)$  We have    $\varphi_{t}(x_{\varepsilon(t)})=\inf_{y\in H} \{f(y)+\frac{\varepsilon(t)}{2}\|y-0\|^{2}\}=f_{\frac{1}{\varepsilon(t)}}(0).$ \\
	 Since  $ \dfrac{d}{d\lambda}f_{\lambda}(x)=-\frac{1}{2}\|\nabla f_{\lambda}(x)\|^{2}, $ (see \cite[Lemma 3.27]{Att-book}, \cite[Corollary 6.2]{Att2}),    we have:
	$$ \dfrac{d}{dt}f_{\lambda(t)}(x)=-\frac{\dot{\lambda}(t)}{2}\|\nabla f_{\lambda(t)}(x)\|^{2} .$$ 
	Therefore, 
	\begin{equation}\label{4}
	 \dfrac{d}{dt} \varphi_{t}(x_{\varepsilon(t)})=\dfrac{d}{dt}\left(f_{\frac{1}{\varepsilon(t)}}(0)\right)=\frac{1}{2}\dfrac{\dot{\varepsilon}(t)}{\varepsilon^{2}(t)}\|\nabla f_{\frac{1}{\varepsilon(t)}}(0)\|^{2}.
	\end{equation}
On the other hand, we have 
$$ \nabla\varphi_{t}((x_{\varepsilon(t)})=0\Longleftrightarrow \nabla f(x_{\varepsilon(t)})+\varepsilon(t)x_{\varepsilon(t)}=0\Longleftrightarrow x_{\varepsilon(t)}=J^{f}_{\frac{1}{\varepsilon(t)}}(0).$$	
Since  $\nabla f_{\frac{1}{\varepsilon(t)}}(0)=\varepsilon(t)\left(0- J^{f}_{\frac{1}{\varepsilon(t)}}(0)\right), $  we get  
$ \nabla f_{\frac{1}{\varepsilon(t)}}(0)=-\varepsilon(t)x_{\varepsilon(t)} $.  This combined with \eqref{4} gives
$$ \dfrac{d}{dt} \varphi_{t}(x_{\varepsilon(t)})=\frac{1}{2}\dot{\varepsilon}(t)\|x_{\varepsilon(t)}\|^{2}.$$ 
\item[$ii)$] We have 
$$ -\varepsilon(t)x_{\varepsilon(t)}=\nabla f(x_{\varepsilon(t)})\quad\hbox{and}\quad -\varepsilon(t+h)x_{\varepsilon(t+h)}=\nabla f(x_{\varepsilon(t+h)}). $$
According to the monotonicity  of $\nabla f,$ we have
$$ \langle \varepsilon(t)x_{\varepsilon(t)}-\varepsilon(t+h)x_{\varepsilon(t+h)},x_{\varepsilon(t+h)}- x_{\varepsilon(t)}  \rangle \geq 0 ,$$
which implies 
$$ -\varepsilon(t)\|u_{\varepsilon(t+h)}- u_{\varepsilon(t)}\|^{2} + \left(\varepsilon(t)-\varepsilon(t+h)\right) \langle u_{\varepsilon(t+h)},u_{\varepsilon(t+h)}- u_{\varepsilon(t)}  \rangle \geq 0 .$$
After division by $h^{2},$ we obtain 
$$ \dfrac{ \left(\varepsilon(t)-\varepsilon(t+h)\right)}{h} \left\langle x_{\varepsilon(t+h)},\dfrac{x_{\varepsilon(t+h)}- x_{\varepsilon(t)}}{h} \right \rangle \geq \varepsilon(t)\left\|\dfrac{x_{\varepsilon(t+h)}- x_{\varepsilon(t)}}{h}\right\|^{2} .$$
By letting $h \rightarrow 0,$ we get 
 $$-\dot{\varepsilon}(t) \left\langle x_{\varepsilon(t)},\dfrac{d}{dt}x_{\varepsilon(t)} \right \rangle \geq \varepsilon(t)\left\|\dfrac{d}{dt}x_{\varepsilon(t)} \right\|^{2} ,$$
which completes the proof. 
\end{proof}
\subsection{Main result}
 Given a general parameter $\varepsilon(\cdot)$, let's proceed with the Lyapunov analysis.

\begin{theorem}\label{strong-conv-thm-b}
Suppose that $f: \cH \to \mathbb R$ is a convex function of class ${\mathcal C}^1$.
Let  $x(\cdot): [t_0, + \infty[ \to \cH$ be a solution trajectory of the system {\rm(TRIGS)} with $\delta>0$
\begin{equation*}
 \ddot{x}(t) + \d\sqrt{\e(t)}  \dot{x}(t) + \nabla f (x(t)) + \varepsilon (t) x(t) =0.
\end{equation*}
Let us assume that there exists $a,c>1$  and $t_1 \geq t_0$ such that for all $t \geq t_1 ,$
$$( \mathcal{H}_1) \qquad \dfrac{d}{dt}\left(\dfrac{1}{\sqrt{\varepsilon(t)}}\right) \leq 
\min\left(2\lambda-\delta\; , \; \delta-\frac{a+1}{a}\lambda\right),
$$
where $\lambda$ is  such that
$$\frac{\delta}{2}< \lambda <\frac{a}{a+1}\delta \; \mbox{ for }  0 <\delta \leq 2-\frac{1}{c}, \quad 
\frac12\left(\delta+\frac{1}{c}+\sqrt{(\delta+\frac{1}{c})^{2}-4}\right)
< \lambda <\frac{a}{a+1}\delta \;  \mbox{ for } \delta > 2-\frac{1}{c}.
$$
 Then, the following property holds:  
\begin{equation}\label{Lyap-basic1}
  E(t)\leq \dfrac{\|x^{*}\|^{2}}{2}\dfrac{\displaystyle \int_{t_1}^{t}\left[\left((\lambda c+a)\lambda\dfrac{\dot{\varepsilon}^{2}(s)}{\varepsilon^{\frac{3}{2}}(s)}-\dot{\varepsilon}(s)\right) \gamma(s) \right]ds}{\gamma(t)}+ \dfrac{\gamma(t_1)E(t_1)}{\gamma(t)}
\end{equation}
  where $\gamma(t)=\exp\left(\displaystyle \int_{t_1}^{t} \mu(s)ds\right),$ and $ \mu(t)=-\dfrac{\dot{\varepsilon}(t)}{2\varepsilon(t)}+(\delta-\lambda)\sqrt{\varepsilon(t)}$.
\end{theorem} %
\begin{proof}
According to the classical derivation  chain rule and Lemma \ref{lem1} $i)$, the differentiation of the function $E(\cdot)$ gives
\begin{equation}\label{6}
\begin{array}{lll}
\dot{E}(t)&=  &  \langle \nabla\varphi_{t}(x(t)),\dot{x}(t)\rangle+\frac{1}{2}\dot{\varepsilon}(t)\|x(t)\|^{2}-\frac{1}{2}\dot{\varepsilon}(t)\|x_{\varepsilon(t)}\|^{2}+ \langle\dot{v}(t),v(t)\rangle.
\end{array} 
\end{equation}
Using the constitutive equation  \eqref{1}, we have $$ 
\begin{array}{lll}
\dot{v}(t)&=  & \dot{\tau}(t)\left(x(t)-x_{\varepsilon(t)}\right)+ \tau(t)\dot{x}(t)-\tau(t)\dfrac{d}{dt}x_{\varepsilon(t)}+\ddot{x}(t), \\ 
& = & \dot{\tau}(t)\left(x(t)-x_{\varepsilon(t)}\right)+ \tau(t)\dot{x}(t)-\tau(t)\dfrac{d}{dt}x_{\varepsilon(t)}-\delta\sqrt{\varepsilon(t)}\dot{x}(t)-\nabla\varphi_{t}(x(t))\\
& = &\dot{\tau}(t)\left(x(t)-x_{\varepsilon(t)}\right)+ \left(\tau(t)-\delta\sqrt{\varepsilon(t)}\right)\dot{x}(t)-\tau(t)\dfrac{d}{dt}x_{\varepsilon(t)}-\nabla\varphi_{t}(x(t)).
\end{array} $$
Therefore,
\begin{eqnarray}\label{7}
	\langle\dot{v}(t),v(t)\rangle&=&  \left\langle\dot{\tau}(t)\left(x(t)-x_{\varepsilon(t)}\right)+ \left(\tau(t)-\delta\sqrt{\varepsilon(t)}\right)\dot{x}(t),\tau(t)(x(t)-x_{\varepsilon(t)})+\dot{x}(t)\right\rangle  \nonumber \\ 
	&+&  \left\langle-\tau(t)\dfrac{d}{dt}x_{\varepsilon(t)}-\nabla\varphi_{t}(x(t)),\tau(t)(x(t)-x_{\varepsilon(t)})+\dot{x}(t)\right\rangle  \nonumber  \\ 
&=& \dot{\tau}(t) \tau(t)\|x(t)-x_{\varepsilon(t)}\|^{2}+\left(\dot{\tau}(t)+\tau^{2}(t)-\delta\tau(t)\sqrt{\varepsilon(t)}\right)\langle x(t)-x_{\varepsilon(t)},\dot{x}(t)\rangle \nonumber \\
	&+& \left(\tau(t)-\delta\sqrt{\varepsilon(t)}\right)\|\dot{x}(t)\|^{2}-\tau^{2}(t)\langle\dfrac{d}{dt} x_{\varepsilon(t)},x(t)-x_{\varepsilon(t)}\rangle-\tau(t)\langle\dfrac{d}{dt} x_{\varepsilon(t)},\dot{x}(t)\rangle  \nonumber\\
		& & \underbrace{-\tau(t)\langle\nabla\varphi_{t}(x(t)),x(t)-x_{\varepsilon(t)})\rangle}_{=A_0} -\langle\nabla\varphi_{t}(x(t)),\dot{x}(t)\rangle,
\end{eqnarray} 
Since $\varphi_{t}$ is strongly convex, we have
$$\varphi_{t}(x_{\varepsilon(t)})-\varphi_{t}(x(t))\geq \left\langle \nabla\varphi_{t}(x(t)),x_{\varepsilon(t)}-x(t) \right\rangle+\frac{\varepsilon(t)}{2} \| x(t)-x_{\varepsilon(t)}\|^{2}.$$
Using the above inequality, we get  
\begin{equation}\label{8}
A_0=-\tau (t)\langle\nabla\varphi_{t}(x(t)),x(t)-x_{\varepsilon(t)})\rangle	 \leq   -\tau(t)\left(\varphi_{t}(x(t))-\varphi_{t}(x_{\varepsilon(t)})\right)-\dfrac{\tau(t)\varepsilon(t)}{2}\| x(t)-x_{\varepsilon(t)}\|^{2}. 
\end{equation}
Moreover, we have for any $a>1,$ 
$$-\tau(t)\left\langle\dfrac{d}{dt} x_{\varepsilon(t)},\dot{x}(t)\right\rangle \leq\dfrac{\tau(t)}{2a}\|\dot{x}(t)\|^{2} +\dfrac{a\tau(t)}{2}\left\|\dfrac{d}{dt} x_{\varepsilon(t)}\right\|^{2}
$$
  and  for any $b>0$
  $$-\tau^{2}(t)\left\langle\dfrac{d}{dt} x_{\varepsilon(t)},x(t)-x_{\varepsilon(t)}\right\rangle \leq \dfrac{b\tau(t)}{2}\left\|\dfrac{d}{dt} x_{\varepsilon(t)}\right\|^{2} +\dfrac{\tau^{3}(t)}{2b}\|x(t)-x_{\varepsilon(t)}\|^{2}.$$
By combining the two above inequalities with \eqref{6}, \eqref{7} and \eqref{8}, and after reduction we obtain
\begin{eqnarray}\label{9}
\dot{E}(t)&\leq  & -\tau(t)\left(\varphi_{t}(x(t))-\varphi_{t}(x_{\varepsilon(t)})\right)+\frac{1}{2}\dot{\varepsilon}(t)\|x(t)\|^{2}-\frac{1}{2}\dot{\varepsilon}(t)\|x_{\varepsilon(t)}\|^{2} \nonumber \\
&& +\left[\dot{\tau}(t) \tau(t)-\dfrac{\tau(t)\varepsilon(t)}{2}\right]\|x(t)-x_{\varepsilon(t)}\|^{2} \nonumber\\ 
& &+\left(\dot{\tau}(t)+\tau^{2}(t)-\delta\tau(t)\sqrt{\varepsilon(t)}\right)\langle x(t)-x_{\varepsilon(t)},\dot{x}(t)\rangle+\left(\tau(t)-\delta\sqrt{\varepsilon(t)}\right)\|\dot{x}(t)\|^{2} \nonumber\\
&&-\tau^{2}(t)\langle\dfrac{d}{dt} x_{\varepsilon(t)},x(t)-x_{\varepsilon(t)} \rangle-\tau(t) \left\langle\dfrac{d}{dt} x_{\varepsilon(t)},\dot{x}(t)\right\rangle
\nonumber \\
&\leq &-\tau(t)\left(\varphi_{t}(x(t))-\varphi_{t}(x_{\varepsilon(t)})\right)+\frac{1}{2}\left[(b+a)\tau(t)\right] \left\|\dfrac{d}{dt}x_{\varepsilon(t)}\right\|^{2} \nonumber\\
&&+\frac{\dot{\varepsilon}(t)}{2}\|x(t)\|^{2} -\dfrac{\dot{\varepsilon}(t)}{2}\|x_{\varepsilon(t)}\|^{2}+\left((1+\frac{1}{2a})\tau(t)-\delta\sqrt{\varepsilon(t)}\right)\|\dot{x}(t)\|^{2}\nonumber \\
&&+\left[\dot{\tau}(t) \tau(t)-\dfrac{\tau(t)\varepsilon(t)}{2}+\dfrac{\tau^{3}(t)}{2b}\right]\|x(t)-x_{\varepsilon(t)}\|^{2} \nonumber \\
&&+\left(\dot{\tau}(t)+\tau^{2}(t)-\delta\tau(t)\sqrt{\varepsilon(t)}\right)\langle x(t)-x_{\varepsilon(t)},\dot{x}(t)\rangle . 
\end{eqnarray}
On the other hand, for a positive function $\mu(t),$ we have
\begin{equation}\label{10}
\begin{array}{lll}
\mu(t)E(t)& = & \mu(t)\left(\varphi_{t}(x(t))-\varphi_{t}(x_{\varepsilon(t)})\right) +\dfrac{\mu(t)}{2}\|v(t)\|^{2}  \\ 
& = & \mu(t)\left(\varphi_{t}(x(t))-\varphi_{t}(x_{\varepsilon(t)})\right) +\dfrac{\mu(t)\tau^{2}(t)}{2}\|x(t)-x_{\varepsilon(t)}\|^{2}+\dfrac{\mu(t)}{2}\|\dot{x}(t)\|^{2}\\
&&+\mu(t)\tau(t)\langle x(t)-x_{\varepsilon(t)}, \dot{x}(t) \rangle.
\end{array} 
\end{equation}
By adding \eqref{9} and \eqref{10}, we get
\begin{eqnarray}\label{11}
\dot{E}(t)+\mu(t)E(t)&\leq &\left(\mu(t)-\tau(t)\right)\left(\varphi_{t}(x(t))-\varphi_{t}(x_{\varepsilon(t)})\right)+\frac{1}{2}\left[(b+a)\tau(t)\right]\|\dfrac{d}{dt}x_{\varepsilon(t)}\|^{2} \nonumber \\
&&-\frac{1}{2}\dot{\varepsilon}(t)\|x_{\varepsilon(t)}\|^{2} +\frac{\dot{\varepsilon}(t)}{2}\|x(t)\|^{2} +\left((1+\frac{1}{2a})\tau(t)-\delta\sqrt{\varepsilon(t)}+\frac{\mu(t)}{2}\right)\|\dot{x}(t)\|^{2} \nonumber\\
&&+\left[\dot{\tau}(t) \tau(t)-\dfrac{\tau(t)\varepsilon(t)}{2}+\dfrac{\tau^{3}(t)}{2b}+\dfrac{\mu(t)\tau^{2}(t)}{2}\right]\|x(t)-x_{\varepsilon(t)}\|^{2} \nonumber\\
&&+\left(\dot{\tau}(t)+\tau^{2}(t)-\delta\tau(t)\sqrt{\varepsilon(t)}+\mu(t)\tau(t)\right)\langle x(t)-x_{\varepsilon(t)},\dot{x}(t)\rangle.
\end{eqnarray} 
As we do not know a priori the sign of $\langle x(t)-x_{\varepsilon(t)},\dot{x}(t)\rangle,$ we take the coefficient in front of this term equal to zero, which gives
\begin{equation}\label{12}
\dot{\tau}(t)+\tau^{2}(t)-\delta\tau(t)\sqrt{\varepsilon(t)}+\mu(t)\tau(t)=0.
\end{equation}
Let us make the following choice of the time dependent parameter
$\tau(t)$ introduced in \eqref{3b} (indeed, it is a key ingredient of our Lyapunov analysis)
$$\tau(t) = \lambda \sqrt{\varepsilon(t)},$$
where $\lambda$ is a positive parameter to be fixed.
Then,  the relation \eqref{12} can be equivalently written
$$ \mu(t)=-\dfrac{\dot{\varepsilon}(t)}{2\varepsilon(t)}+(\delta-\lambda)\sqrt{\varepsilon(t)}.$$
According to this choice for $\mu( t )$ and  $\tau ( t ),$ and neglecting the term $\frac{\dot{\varepsilon}(t)}{2}\|x(t)\|^{2}$ which is less than or equal to zero,  the inequality \eqref{11} becomes
\begin{equation}\label{13}
\begin{array}{lll}
\dot{E}(t)+\mu(t)E(t)&\leq &\dfrac{1}{2\varepsilon(t)}\left(-\dot{\varepsilon}(t)+2(\delta-2\lambda)\varepsilon^{\frac{3}{2}}(t)\right)\left(\varphi_{t}(x(t))-\varphi_{t}(x_{\varepsilon(t)})\right) \\
&&+\dfrac{1}{2}\left[(b+a)\lambda\sqrt{\varepsilon(t)}\right] \left\|\dfrac{d}{dt}x_{\varepsilon(t)}\right\|^{2}-\frac{1}{2}\dot{\varepsilon}(t)\|x_{\varepsilon(t)}\|^{2} \\
&&+\dfrac{1}{4\varepsilon(t)}\left[2\left((1+\frac{1}{a})\lambda-\delta\right)\varepsilon^{\frac{3}{2}}(t)-\dot{\varepsilon}(t)\right]\|\dot{x}(t)\|^{2}\\
&&+\dfrac{\lambda}{4}\left[\lambda\dot{\varepsilon}(t)+2\left(\delta \lambda-\lambda^{2}-1\right)\varepsilon^{\frac{3}{2}}(t)+2\frac{\lambda^{2}}{b}\varepsilon^{\frac{3}{2}}(t)\right]\|x(t)-x_{\varepsilon(t)}\|^{2}.
\end{array} 
\end{equation}
According to item $ii)$ of Lemma \ref{lem1}, and inequality
\eqref{2a}
\begin{center}
 $\left\|\dfrac{d}{dt}x_{\varepsilon(t)}\right\|^{2}\leq \dfrac{\dot{\varepsilon}^{2}(t)}{\varepsilon^{2}(t)}\|x_{\varepsilon(t)}\|^{2}\leq \dfrac{\dot{\varepsilon}^{2}(t)}{\varepsilon^{2}(t)}\|x^{*}\|^{2} .$
\end{center}
 Using again  inequality
\eqref{2a}, and the fact that $ -\dot{\varepsilon}(t) \geq 0$, we have 
\begin{center}
$
-\frac{1}{2}\dot{\varepsilon}(t)\|x_{\varepsilon(t)}\|^{2} \leq
-\frac{1}{2} \dot{\varepsilon}(t)  \|x^{*}\|^{2}.
$
\end{center}
By inserting the two inequalities above in \eqref{13}, we obtain
\begin{equation}\label{14a}
\begin{array}{lll}
\dot{E}(t)+\mu(t)E(t)&\leq &\dfrac{1}{2\varepsilon(t)}\left(-\dot{\varepsilon}(t)+2(\delta-2\lambda)\varepsilon^{\frac{3}{2}}(t)\right)\left(\varphi_{t}(x(t))-\varphi_{t}(x_{\varepsilon(t)})\right) \\
&&+\dfrac{1}{2}\left[(b+a)\lambda\dfrac{\dot{\varepsilon}^{2}(t)}{\varepsilon^{\frac{3}{2}}(t)}-\dot{\varepsilon}(t)\right]\|x^{*}\|^{2}\\
&&+\dfrac{1}{4\varepsilon(t)} \left[2\left((1+\frac{1}{a})\lambda-\delta\right)\varepsilon^{\frac{3}{2}}(t)-\dot{\varepsilon}(t)\right]\|\dot{x}(t)\|^{2}\\
&&+\dfrac{\lambda}{4} \left[\lambda\dot{\varepsilon}(t)+2\left(\delta \lambda-\lambda^{2}-1\right)\varepsilon^{\frac{3}{2}}(t)+2\frac{\lambda^{2}}{b}\varepsilon^{\frac{3}{2}}(t)\right] \|x(t)-x_{\varepsilon(t)}\|^{2}.
\end{array} 
\end{equation}
By taking $b=c\lambda,$ with $c>1,$ we get
\begin{equation}\label{14}
\begin{array}{lll}
\dot{E}(t)+\mu(t)E(t)&\leq &\dfrac{1}{2\varepsilon(t)}\underbrace{\left(-\dot{\varepsilon}(t)+2(\delta-2\lambda)\varepsilon^{\frac{3}{2}}(t)\right)}_{=A}\left(\varphi_{t}(x(t))-\varphi_{t}(x_{\varepsilon(t)})\right) \\
&&+\dfrac{1}{2}\left[(\lambda c+a)\lambda\dfrac{\dot{\varepsilon}^{2}(t)}{\varepsilon^{\frac{3}{2}}(t)}-\dot{\varepsilon}(t)\right]\|x^{*}\|^{2}\\
&&+\dfrac{1}{4\varepsilon(t)}\underbrace{\left[2\left((1+\frac{1}{a})\lambda-\delta\right)\varepsilon^{\frac{3}{2}}(t)-\dot{\varepsilon}(t)\right]}_{=B}\|\dot{x}(t)\|^{2}\\
&&+\dfrac{\lambda}{4}\underbrace{\left[\lambda\dot{\varepsilon}(t)+2\left((\delta+\frac{1}{c}) \lambda-\lambda^{2}-1\right)\varepsilon^{\frac{3}{2}}(t)\right]}_{=C}\|x(t)-x_{\varepsilon(t)}\|^{2} .
\end{array} 
\end{equation}
 We are looking for sufficient conditions on the parameters $\lambda$ and $\varepsilon (t)$ which make $A$, $B$, and $C$ less than or equal to zero.
It is here that the hypothesis $(\mathcal{H}_1)$ formulated in the statement of the Theorem \ref{strong-conv-thm-b} intervenes. It is recalled below for the convenience of the reader 
 
 \medskip

$(\mathcal{H}_1)$ \quad  There exists $a,c>1$  and $t_1 \geq t_0$ such that for all $t \geq t_1 ,$
$$\dfrac{d}{dt}\left(\dfrac{1}{\sqrt{\varepsilon}(t)}\right) \leq \min\left(2\lambda-\delta\; ,\; \delta-\frac{a+1}{a}\lambda\right),$$
where $\lambda$ is  such that
$\frac{\delta}{2}< \lambda <\frac{a}{a+1}\delta$ for $0 <\delta \leq 2-\frac{1}{c}$
and 
$\frac12\left(\delta+\frac{1}{c}+\sqrt{(\delta+\frac{1}{c})^{2}-4}\right) < \lambda <\frac{a}{a+1}\delta $
 for $\delta > 2-\frac{1}{c}.$

\bigskip

 
\medskip

\noindent Clearly,  condition 
 $(\mathcal{H}_1)$ is equivalent to
 $$\dfrac{d}{dt}\left(\dfrac{1}{\sqrt{\varepsilon}(t)}\right) \leq 2\lambda-\delta
 \quad  and \quad
\dfrac{d}{dt}\left(\dfrac{1}{\sqrt{\varepsilon}(t)}\right)\leq \delta-\frac{a+1}{a}\lambda.$$
\noindent Under condition   $(\mathcal{H}_1)$ we immediately obtain that $A$ and $B$ are less than or equal to zero:

\medskip

 	$\bullet$ \;  $A=-\dot{\varepsilon}(t)+2(\delta-2\lambda)\varepsilon^{\frac{3}{2}}(t)=2\varepsilon^{\frac{3}{2}}(t)\left[\dfrac{d}{dt}\left(\dfrac{1}{\sqrt{\varepsilon(t)}}\right)+\delta-2\lambda \right]\leq 0;$
 	
 	\medskip
 	
 	$\bullet$ \; $B=2\left((1+\frac{1}{a})\lambda-\delta\right)\varepsilon^{\frac{3}{2}}(t)-\dot{\varepsilon}(t)=
	2\varepsilon^{\frac{3}{2}}(t)\left[\dfrac{d}{dt}\left(\dfrac{1}{\sqrt{\varepsilon(t)}}\right) + \dfrac{
	a+1}{a}\lambda-\delta\right]\leq 0.$

 \medskip
 
  	$\bullet$ \; Let us now examine $C$.
  $$
 	\begin{array}{lll}
 	C&=&\lambda\dot{\varepsilon}(t)+2\left((\delta+\frac{1}{c}) \lambda-\lambda^{2}-1\right)\varepsilon^{\frac{3}{2}}(t)\\
 	&=&2\varepsilon^{\frac{3}{2}}(t)\left[-\lambda\underbrace{\dfrac{d}{dt}\left( \dfrac{1}{\sqrt{\varepsilon(t)}}\right)}_{\geq 0  }+ \left((\delta+\frac{1}{c}) \lambda-\lambda^{2}-1\right) \right].
 	\end{array}
 	$$
 	When $\delta \leq 2-\frac{1}{c}$ we have
$(\delta+\frac{1}{c})\lambda-\lambda^{2} -1 \leq 2\lambda-\lambda^{2} -1 \leq 0.$\\
 When $\delta > 2-\frac{1}{c},$ we have $(\delta+\frac{1}{c})\lambda-\lambda^{2}-1 \leq 0,$ because $\lambda\geq  \frac12\left(\delta+\frac{1}{c}+\sqrt{(\delta+\frac{1}{c})^{2}-4}\right)$.\\
This implies that $C\leq 0.$

\medskip 	
 	
\noindent According to \eqref{14}, under condition   $(\mathcal{H}_1)$ we conclude that
 \begin{equation}\label{15}
 \dot{E}(t)+\mu(t)E(t)\leq \frac{1}{2}\left[(\lambda c+a)\lambda\dfrac{\dot{\varepsilon}^{2}(t)}{\varepsilon^{\frac{3}{2}}(t)}-\dot{\varepsilon}(t)\right]\|x^{*}\|^{2}
 \end{equation}
  By multiplying the inequality above with $\gamma(t)=\exp\left(\displaystyle \int_{t_1}^{t} \mu(s)ds\right),$ we obtain
  \begin{equation}\label{16}
  \dfrac{d}{dt}\left(\gamma(t) E(t)\right)\leq \dfrac{\|x^{*}\|^{2}}{2}\left[(\lambda c+a)\lambda\dfrac{\dot{\varepsilon}^{2}(t)}{\varepsilon^{\frac{3}{2}}(t)}-\dot{\varepsilon}(t)\right] \gamma(t)
  \end{equation}
By integrating \eqref{16} on $[t_1 , t]$, we get

$$ E(t)\leq \dfrac{\|x^{*}\|^{2}}{2}\dfrac{\displaystyle \int_{t_1}^{t}\left[\left((\lambda c+a)\lambda\dfrac{\dot{\varepsilon}^{2}(s)}{\varepsilon^{\frac{3}{2}}(s)}-\dot{\varepsilon}(s)\right) \gamma(s) \right]ds}{\gamma(t)}+ \dfrac{\gamma(t_1)E(t_1)}{\gamma(t)}.
$$

\noindent This completes the proof of the Lyapunov analysis.
\end{proof}

\begin{remark}
Given $\delta>0$, the condition $(\mathcal{H}_1)$ gives the admissible values of the parameters $a>1, \; c>1$ and $\lambda >0$ which enter into the convergence rates obtained in Theorem \ref{strong-conv-thm-b}. 
Let us verify that the inequalities that define the values of these parameters are consistent.
We consider successively  the two cases $\delta < 2$, then $\delta \geq 2$.

\smallskip

a) When $\delta < 2$, we have $\delta < 2 -\frac{1}{c}$ for $c$ sufficiently large.
Because $a>1$,  we have $\frac{1}{2}<\frac{a}{a+1}$, and hence the interval $[\frac{1}{2}\delta, \frac{a}{a+1}\delta]$
is nonempty. Therefore, in this case the conditions are consistent.

\smallskip

b) Suppose now that  $\delta  \geq 2$. Then $\delta > 2-\frac{1}{c}$ for all $c>0$, and we can argue with $c$ arbitrarily large.
 Let us verify that we can find   $a$ and $c$ such that 
\begin{equation}\label{parameters1}
\frac{\delta+\frac{1}{c}+\sqrt{(\delta+\frac{1}{c})^{2}-4}}{2\;} <\frac{a}{a+1}\delta,
\end{equation}
and hence a value of $\delta$ belonging to the corresponding interval.
By letting $c \to +\infty$ and $a \to +\infty$ in the above inequality we obtain
$$
\delta +\sqrt{\delta^2-4} < 2\delta ,
$$
which is equivalent to $\sqrt{\delta^2-4} < \delta$, and hence is satisfied.
By a continuity argument, we obtain  that the inequality 
\eqref{parameters1} is satisfied by taking $a$ and $c$ sufficiently large.\\
Note that, since we are interested in asymptotic results the important point is to get the existence of parameters for which the Lyapunov analysis is valid. If we are interested in complexity results then the precise value of these parameters is important.
\end{remark} 

\begin{remark}
The above argument shows that the controlled decay property 
${\rm(CD)}_{\lambda}$
used in \cite{AL} corresponds to the limiting case
$c \to +\infty$ and $a \to +\infty$ in our condition 
$(\mathcal{H}_1)$.
\end{remark}

\begin{remark}
As in \cite{AL}, our Lyapunov analysis is valid for an arbitrary choice of the parameter $\delta$.
It would be interesting to know what is the best choice for $\delta$.

\end{remark}

\section{Particular cases}\label{sec: particular}
Let's study the case $\varepsilon(t)=\dfrac{1}{t^p}$,
and discuss, according to the value of the parameter $0<p<2$, the convergence rate of  values, and the convergence rate to zero of $\|x(t)-x_{\epsilon(t)}\|$.
The following results were stated in Theorem \ref{thm:model-a}, in the introduction, as model results. We reproduce them here for the convenience of the reader. The point is simply to apply the general Theorem \ref{strong-conv-thm-b} to this particular situation, and to show that the different quantities involved in the convergence results can be computed explicitly.

\begin{theorem}\label{thm:model-aa}
Take $\e(t)=\displaystyle\frac{1}{t^{p} } $,\;   $0<p<2$.
Let $x : [t_0, +\infty[ \to \mathcal{H}$ be a solution trajectory of
$$
\ddot{x}(t) + \frac{\d}{ \displaystyle{t^{\frac{p}{2}}}}\dot{x}(t) + \nabla f\left(x(t) \right)+ \frac{1}{t^p} x(t)=0.
$$
Then, we have  convergence of the values, and strong convergence to the minimum norm solution with the following rates:
\begin{eqnarray}
&& 
f(x(t))-\min_{\cH} f= \mathcal O \left( \displaystyle\frac{1}{t^{p} }   \right) \mbox{ as } \; t \to +\infty;\\
&& \|x(t) -x_{\varepsilon(t)}\|^2=\mathcal{O}\left(\dfrac{1}{ t^{\frac{2-p}2}}\right) \mbox{ as } \; t \to +\infty.
\end{eqnarray}
\end{theorem}
 \begin{proof}
a) \textbf{Convergence rate of the values}:
  With the notations of  Theorem \ref{strong-conv-thm-b},
we have
 $$
 \mu(t)=-\dfrac{\dot{\varepsilon}(t)}{2\varepsilon(t)}+(\delta-\lambda)\sqrt{\varepsilon(t)}=\dfrac{p}{2t}+ \dfrac{\delta -\lambda}{t^{\frac{p}2}}.
 $$
So 
\begin{eqnarray}
\gamma(t)&=&\exp\left(\displaystyle \int_{t_1}^{t} \mu(s)ds\right) =  \left(\dfrac{t}{t_1}\right)^{\frac p2} \exp\left[\frac{2(\delta-\lambda)}{2-p}\left(t^{\frac{2-p}2} - t_1^{\frac{2-p}2} \right)\right] \nonumber\\
& =& 
  C_1t^{\frac p2} \exp\left[\frac{2(\delta-\lambda)}{2-p}t^{\frac{2-p}2} \right] \label{gamma-est-1}
\end{eqnarray} 
 where
 $ 
 C_1=\left(t_1^{\frac p2} \exp\left[\frac{2(\delta-\lambda)}{2-p}t_1^{\frac{2-p}2} \right]\right)^{-1}.
 $
Let us choose the parameters $a, c>1$, 
$\lambda>0$ such that
\begin{center}
$\frac{\delta}{2}< \lambda <\frac{a}{a+1}\delta$ for $ 0 <\delta \leq 2-\frac{1}{c}$ and $\frac{\delta}{2} < \frac12\left(\delta+\frac{1}{c}+\sqrt{(\delta+\frac{1}{c})^{2}-4}\right) < \lambda <\frac{a}{a+1}\delta $ for $\delta > 2-\frac{1}{c}$.
\end{center}
Notice that, in these two cases, we have  $\frac{1}2\delta<\lambda < \frac{a+1}{a}\delta$. Therefore,  $\min\left(2\lambda-\delta\; , \;  \delta-\frac{a+1}{a}\lambda\right)>0$. 
On the other hand, since $p<2$, we have
 $\lim_{t \to +\infty} t^{\frac{p-2}{2}}=0$. So we have, for  $t\geq t_1$ large enough,  
$$
\dfrac{d}{dt}\left(\dfrac{1}{\sqrt{\varepsilon(t)}}\right) =\frac{p}2t^{\frac{p-2}{2}}
\leq \min\left(2\lambda-\delta\; , \;  \delta-\frac{a+1}{a}\lambda\right).
$$
As a consequence, the condition $(\mathcal{H}_1)$ is satisfied.

\smallskip

\noindent According to \eqref{Lyap-basic1}, we have $E(t) \leq E_1(t) \|x^{*}\|^{2} + E_2(t)$ where 
\begin{equation}\label{Lyap-basic1-b}
  E_1(t)=\dfrac{1}{2\gamma(t)}\displaystyle{ \int_{t_1}^{t}\left[\left((\lambda c+a)\lambda\dfrac{\dot{\varepsilon}^{2}(s)}{\varepsilon^{\frac{3}{2}}(s)}-\dot{\varepsilon}(s)\right) \gamma(s) \right]ds}
\end{equation}
and
\begin{equation}\label{Lyap-basic1-c}
  E_2(t)=\dfrac{\gamma(t_1)E(t_1)}{\gamma(t)}.
\end{equation}
According to \eqref{gamma-est-1} we have 
$$
E_2(t) \leq C t^{-\frac p2} \exp\left[-\frac{2(\delta-\lambda)}{2-p}t^{\frac{2-p}2} \right].
$$
Since $0 <p<2$ and $\delta > \lambda$, we have that $E_2(t)$ tends to zero at an exponential rate, as $t \to +\infty$.\\
Thus, we only have to focus on the asymptotic behavior of $E_1 (t)$. Let us simplify the formula
by setting $\lambda_0:=(\lambda c+a)\lambda , \; \delta_0:=\dfrac{2(\delta-\lambda)}{2-p}$ in \eqref{Lyap-basic1}. Replacing $\varepsilon (t)$ and $\gamma(t)$ by their values in
\eqref{Lyap-basic1-b} gives 
 $$
E_1(t) = \dfrac{p}{2 t^{\frac p2} \exp\left(\delta_0t^{\frac{2-p}2} \right)}\displaystyle     \int_{t_1}^{t}\left( \dfrac{\lambda_0p}{s^2 } + \frac{1}{s^{\frac{p+2}2}}\right)\exp\left(\delta_0s^{\frac{2-p}2}\right)ds.
 $$
To compute the above integral, we notice that
$$
\dfrac{d}{ds}\left( \dfrac{1}{\rho s  } \exp\left(\delta_0 s^{\frac{2-p}2}\right)\right)= \left(  -\dfrac{1}{\rho s^2 } + \dfrac{\delta_0(2-p)}{2\rho s^{\frac{p+2}2}}   
\right)\exp\left(\delta_0s^{\frac{2-p}2}\right).
$$
Then, note that, when we set $\rho>0$ such that $\rho <\min \left( 1\;,\; \frac1{a+1}\delta\right)$, we obtain
$$
\begin{array}{lll}
 \dfrac{\lambda_0p}{s^2 } + \dfrac{1}{s^{\frac{p+2}2}} \leq -\dfrac{1}{\rho s^2 } + \dfrac{\delta_0(2-p)}{2\rho s^{\frac{p+2}2}}   
 & \Longleftrightarrow &
 \dfrac{\lambda_0p+\frac1{\rho}}{s^2 }  \leq \left(\dfrac{\delta_0(2-p)}{2\rho}-1\right) \dfrac1{s^{\frac{p+2}2}}
= \dfrac{\frac{\delta-\lambda}{\rho}-1}{s^{\frac{p+2}2}}\\
 & \Longleftrightarrow &
 \dfrac{\lambda_0p+\frac1{\rho}}{s^{\frac{2-p}2} }  \leq \dfrac{\delta-\lambda}{\rho}-1.
 \end{array}
 $$
Let us verify that the last above inequality is satisfied for $s$ large enough. First,  since $0<p<2$, we have   $\frac{2-p}2>0$,
and hence  $\lim_{s \to +\infty} \dfrac{1}{s^{\frac{2-p}2} } =0$. On the other hand, according to $(\mathcal{H}_1)$ we have
$\lambda < \frac{a}{a+1}\delta$. This property combined with the choice of $\rho$ implies
$$
\delta-\rho > \delta - \frac1{a+1}\delta =  \frac{a}{a+1}\delta >  \lambda.
$$
Therefore  $\dfrac{\delta-\lambda}{\rho}-1 >0$.
So,    for $t_1$ large enough
$$
\begin{array}{lll}
E_1(t) &\leq& \dfrac{p}{2 t^{\frac p2} \exp\left(\delta_0t^{\frac{2-p}2} \right)}\displaystyle \int_{t_1}^{t}\left(  -\dfrac{1}{\rho s^2 } + \dfrac{\delta_0(2-p)}{2\rho s^{\frac{p+2}2}}   
\right)\exp\left(\delta_0s^{\frac{2-p}2}\right)ds\\
&=& \dfrac{1}{2  t^{\frac p2} \exp\left(\delta_0t^{\frac{2-p}2} \right)}\displaystyle\int_{t_1}^{t}\dfrac{d}{ds}\left( \dfrac{1}{\rho s } \exp\left(\delta_0s^{\frac{2-p}2}\right)\right)ds\\
&=& \dfrac{p}{ 2\rho t^{\frac{p+2}2}} - \dfrac{p}{ t^{\frac{p}2}\exp\left(\delta_0t^{\frac{2-p}2} \right)}  \dfrac{1}{2\rho t_1 } \exp\left(\delta_0t_1^{\frac{2-p}2}\right)
\leq \dfrac{p}{2\rho t^{\frac{p+2}2}}.
 \end{array}
 $$
Since $E_2 (t)$ has an exponential decay to zero, we deduce that there exists a positive constant $C$ such that for $t$ large enough
$$
E(t) \leq \dfrac{C}{ t^{\frac{p+2}{2}}}.
$$
According to Lemma \ref{lem-basic}, we get
$$
f(x(t))-\min_{\mathcal H}f \leq C \left(\dfrac{1}{ t^{\frac{p+2}{2}}}  + \frac{1}{t^p} \right).
$$
Since $0<p<2,$ we have $p < \frac{p+2}{2}$.  We conclude that
$$
f(x(t))-\min_{\cH} f=\mathcal O \left( \displaystyle{ \frac{1}{t^{p}} }   \right) \mbox{ as } \; t \to +\infty.
$$

b) \textbf{Convergence rate to zero of $\|x(t)-x^*\|$}. 
According to Lemma \ref{lem-basic}, we have
\begin{equation}\label{est:basic1b}
 \|x(t) - x_{\varepsilon(t)}\|^2  \leq \frac{2E(t)}{\varepsilon(t)},
 \end{equation}
and since, for $t$ large enough,
$
E(t) \leq \dfrac{C}{ t^{\frac{p+2}{2}}}
$, we obtain
 $$
  \|x(t)-x_{\varepsilon(t)}\|^{2}=\mathcal{O}\left(\dfrac{1}{ t^{\frac{2-p}2}}\right),
  $$
which completes the proof.
\end{proof}

\begin{remark}
The convergence rate of the values 
$\mathcal O \left( \displaystyle{ \frac{1}{t^{p}} }   \right) $ obtained in Theorem 
\ref{thm:model-aa} notably improves the result obtained in 
\cite{AL}, where the convergence rate was of order
$\mathcal O \left (  \dfrac{1}{ t^{\frac{3p -2}{2}}}  \right)$.
Indeed, for $p<2$ we  have $p > \frac{3p -2}{2} $.
In addition, we have obtained that for any $0<p<2$, for any trajectory of (TRIGS), we have strong convergence of the trajectory to the minimum norm solution, as time $t$ tends toward $+\infty$. In \cite{AL} it was only obtained 
$\liminf_{t\to +\infty} \|x(t)-x^*\| =0$. 
\end{remark}

\begin{remark}
A close look at the proof of Theorem \ref{thm:model-aa} shows that the convergence rate of the values is still valid in the case $p=2$. Precisely, by taking $\varepsilon(t)= \frac{c}{t^2}$
with $c$ sufficiently small, we have that the condition 
$( \mathcal{H}_1)$ is satisfied, and hence 
$$f(x(t))-\min_{\cH} f= \mathcal O \left( \displaystyle\frac{1}{t^{2} }   \right) \mbox{ as } \; t \to +\infty.$$
\end{remark}

\begin{remark}
As a key ingredient of our proof of the  strong convergence of the trajectories of (TRIGS) to the element of minimum norm of $S$, we use the function
$
h(t):= \demi \|x(t) - x_{\varepsilon(t)}\|^2,
$
and show that $\lim_{t\to +\infty} h(t)=0$.
This strategy which consists in showing that the trajectory is not too far from the viscosity curve $\varepsilon \mapsto x_\varepsilon $  was already present in the approach developed by Attouch and Cominetti in \cite{AttCom} for the study of similar questions in the case of the steepest descent method.
\end{remark}

\subsection{Trade-off between the convergence rate of values and  trajectories}
The following elementary diagram shows  the respective evolution  as $p$ varies of the convergence rate of the values, and the convergence rate of $\|x(t)-x_{\varepsilon(t)}\|^{2}$.

\setlength{\unitlength}{12cm}
\begin{picture}(0.7,0.7)(-0.35, 0.0)
\put(0.55,0.26){$p$}
\put(0.03,0.26){0}
\put(0.345,0.26){2}
\put(0.16,0.26){$\frac{2}{3}$}
\put(0.03,0.585){2}

\put(0.038,0.39){$\frac{2}{3}$}

\multiput(0.065,0.6)(0.02,0){15}
{\line(1,0){0.01}}

\multiput(0.065,0.4)(0.02,0){5}
{\line(1,0){0.01}}

\multiput(0.36,0.3)(0,0.021){14}
{\line(0,1){0.01}}

\multiput(0.165,0.3)(0,0.021){5}
{\line(0,1){0.01}}

\put(-0.01,0.3){\vector(1,0){0.6}}
\put(0.063,0.23){\vector(0,1){0.42}}
\linethickness{0.4mm}
\tcr{
\put(0.05,0.3){\line(1,1){0.3}}
}
\tcb{
\put(0.04,0.453){\line(4,-2){0.30}}
}
\put(-0.25,0.16){$ f (x(t))-\min_{\mathcal H} f = \mathcal O \left( \displaystyle{ \frac{1}{t^{p}} } \right)$ (red color),
\quad $
  \|x(t)-x_{\varepsilon(t)}\|^{2}=\mathcal{O}\left(\dfrac{1}{ t^{\frac{2-p}2}}\right)
  $(blue color)} 
\end{picture}

We observe that $p=\frac{2}{3}$ is a good compromise between these two antagonist properties. Let us state the corresponding result  below.
\begin{corollary}\label{thm:model-b}
Take $\e(t)=\displaystyle{\frac{1}{t^{\frac{2}{3}}}} $, $\delta >0$.
Let $x : [t_0, +\infty[ \to \mathcal{H}$ be a  solution trajectory of
$$
\ddot{x}(t) + \displaystyle{ \frac{\d}{t^{\frac{1}{3}}}}\dot{x}(t) + \nabla f\left(x(t) \right)+ \displaystyle{\frac{1}{t^{\frac{2}{3}}} }x(t)=0.
$$
Then, we have  convergence of the values, and strong convergence to the minimum norm solution with the following rates:
$$
f(x(t))-\min_{\cH} f= \mathcal O \left( \displaystyle{ \frac{1}{t^{\frac{2}{3}}}  }\right) \mbox{ and } \;\;
\| x(t) - x_{\varepsilon (t)} \|^2 = \mathcal O \left( \displaystyle{ \frac{1}{t^{\frac{2}{3}}} }  \right) \mbox{ as } \; t \to +\infty.
$$
\end{corollary}

\smallskip

b) Another interesting case is to take $p<2$, with $p$ close to $2$. In this case, we have a convergence rate of the values which is arbitrarily close to the convergence rate of the accelerated gradient method of Nesterov, with a guarantee of strong convergence towards the minimum norm solution. 
The case $p=2$  has been studied extensively in \cite{AL}.
The strong convergence of the trajectories to the minimum norm solution is an open question in the case $p=2$.

\smallskip
 
c) Estimating the convergence rate of $x(t)$ to $x^*$
 relies on getting informations about  the viscosity trajectory $\epsilon \mapsto x_{\varepsilon}$, and how fast $x_{\varepsilon}$ converges to $x^*$ as $\varepsilon \to 0$.
This is a difficult problem because the viscosity trajectory  can have an infinite length, as  Torralba showed in  \cite{Torralba}.
His counterexample involves the construction of a convex function from its sub-level sets, and relates to a convex function whose sub-level sets vary greatly. Our analysis focuses on general $ f $, \ie the worst case. This suggests that, under good geometric properties of $ f $, such as the Kurdyka-Lojasiewicz property, one should be able to obtain better results; see also \cite{AttCom} where it is shown the importance of this kind of question for the coupling of the steepest descent with Tikhonov approximation.

\section{Non-smooth case}\label{non-smooth}
Let us extend the  previous results to the case of a proper lower semicontinuous and convex function $f: \cH \to \R \cup \left\lbrace +\infty \right\rbrace$.
We rely on the basic properties of the Moreau envelope. Recall that, for any $\theta >0$, $f_{\theta}: \cH \to \R$  is defined by 
\begin{equation}\label{def:prox}
f_{\theta} (x) = \min_{\xi \in \cH} \left\lbrace f (\xi) + \frac{1}{2 \theta} \| x - \xi\| ^2   \right\rbrace, \quad \text{for any $x\in \cH$.} 
\end{equation}
Then  $f_{\theta} $ is a convex differentiable function, whose gradient is $\theta^{-1}$-Lipschitz continuous, and such that $\min_{\cH} f= \min_{\cH} f_{\theta}$, \; $\argmin_{\cH} f_{\theta} = \argmin_{\cH} f$.
Denoting by $\prox_{\theta f}(x)$ the unique point where the minimum value is achieved in \eqref{def:prox},
let us recall the following classical formulae:

\smallskip

\begin{enumerate}
\item $f_{\theta} (x)= f(\prox_{ \theta f}(x)) + \frac{1}{2\theta} \|x-\prox_{\theta f}(x)\| ^2$. \vspace{2mm}
\item $\nabla f_{\theta} (x)= \frac{1}{\theta} \left( x-\prox_{ \theta f}(x) \right)$. 
\end{enumerate}

\smallskip

The interested reader may refer to \cite{BC,Bre1} for a comprehensive treatment of the Moreau envelope in a Hilbert setting. Since the set of minimizers is preserved by taking the Moreau envelope, the idea is to replace $f$ by $f_{\theta} $ in the inertial dynamic (TRIGS).  Then, (TRIGS) applied to $f_{\theta}$ now reads
\begin{equation*}
\mbox{(TRIGS)}_{\theta} \quad \ddot{x}(t) + \d\sqrt{\e(t)}  \dot{x}(t) + \nabla f_{\theta} (x(t)) + \varepsilon (t) x(t) =0.
\end{equation*}
Clearly, since  $f_{\theta} $ is continuously differentiable, the hypothesis $( \mathcal{H}_0)$ are satisfied by the above dynamic.
By applying Theorem \ref{thm:model-aa} to  $\mbox{(TRIGS)}_{\theta}$, we get the following result, which provides convergence rate of the values and strong convergence to the minimum norm solution.

\begin{theorem} 
Let $f: \cH \to \R \cup \left\lbrace +\infty \right\rbrace$ be a convex, lower semicontinuous, proper function. 
Take $\e(t)=1/t^p$,\;   $0<p<2$, and $\theta >0$.
Let $x : [t_0, +\infty[ \to \mathcal{H}$ be a solution trajectory of $\mbox{\rm(TRIGS)}_{\theta}$, \ie
$$
\ddot{x}(t) +  \frac{\d}{t^{\frac{p}{2}}}\dot{x}(t) + \nabla f_{\theta}\left(x(t) \right)+ \frac{1}{t^p} x(t)=0.
$$
Then, we have the following  convergence rates: as $t \to +\infty$
\begin{eqnarray}
&& f({\prox}_{\theta f}(x(t)))-\min_{\cH}f= \mathcal O \left( \displaystyle{ \frac{1}{t^{p}} } \right);\\
&&   \|x(t) - {\prox}_{ \theta f}(x(t)))\|^2  =\mathcal O \left( \displaystyle{ \frac{1}{t^{p} }  } \right);      \\
&& \|x(t) -x_{\varepsilon(t)}\|^2=\mathcal{O}\left(\dfrac{1}{ t^{\frac{2-p}2}}\right).
\end{eqnarray}
\end{theorem}
\begin{proof}
By applying Theorem \ref{thm:model-aa} to the function $f_{\theta} $, and since $\inf f_{\theta} =\inf f$, we get
\begin{eqnarray}
&& 
f_{\theta}(x(t))-\min_{\cH} f= \mathcal O \left( \displaystyle{ \frac{1}{t^{p} } }  \right) \mbox{ as } \; t \to +\infty\\
&& \|x(t) -x_{\varepsilon(t)}\|^2=\mathcal{O}\left(\dfrac{1}{ t^{\frac{2-p}2}}\right) \mbox{ as } \; t \to +\infty.
\end{eqnarray}
According to $f_{\theta} (x)-\min_{\cH} f =\Big( f(\prox_{ \theta f}(x))-\min_{\cH} f \Big) + \frac{1}{2\theta} \|x-\prox_{\theta f}(x)\| ^2$, we get
$$
f({\prox}_{\theta f}(x(t)))-\min_{\cH} f \leq f_{\theta}(x(t))-\min_{\cH} f= \mathcal O \left( \displaystyle{ \frac{1}{t^{p} }}   \right),
$$
which gives the claims.
\end{proof}
\begin{remark}

The above result suggests that, in the case of a nonsmooth convex function  $f: \cH \to \R \cup \left\lbrace +\infty \right\rbrace$, the corresponding proximal algorithms will inherit the  convergence properties of the continuous dynamic (TRIALS).
When considering  convex minimization problems with additive structure
$\min\{ f+g \}$ with $f$ smooth and $g$ nonsmooth, it is  in general difficult to compute the proximal mapping of $f+g$.
A common device  then consists of using a splitting method, and writing the minimization problem as the fixed point problem
$Tx=x$ where, given $\theta >0$
$$
Tx = \mbox{prox}_{\theta g} \left( x- \theta \nabla f(x)\right).
$$
Under appropriate conditions, $T$ is an averaged nonexpansive operator \cite{BC}, so
the associated iterative method (proximal gradient method) converges to a fixed point of $T$, and therefore of the initial minimization problem.
In our context, this  naturally leads to studying the inertial system
\begin{equation*}
\ddot{x}(t) + \d\sqrt{\e(t)}  \dot{x}(t) + (I-T) (x(t)) + \varepsilon (t) x(t) =0.
\end{equation*}
Many  properties of the Tikhonov approximation are still valid for maximally monotone operators, which allows to expect good convergence properties for the above system. This is a subject for further research.
\end{remark}

\section{Numerical illustration}\label{num}
Let us illustrate our results in the following elementary situation. Take
$\cH = \R^{20}$ equipped with the classical Euclidean structure, 
and $f :\R^{20} \to \R^+$ is given by 
$$ f(x_1,\cdots,x_{20}) := \dfrac12\displaystyle\sum_{i=1}^{10}\left(x_{2i-1}+x_{2i}-1\right)^2 .$$
The function $f$ is  convex, but not strongly convex. 
In this case, the solution set 
$S$ is the entire affine subspace $\{x\in\R^{20} : x_{2i-1}+x_{2i}-1=0, \text{ for all }i=1,\cdots,10\}$, and the projection of the origin onto $S$ is given by $x^* = (\demi ,\cdots , \demi)$.
\begin{figure}
  \includegraphics[width=1.05\textwidth]{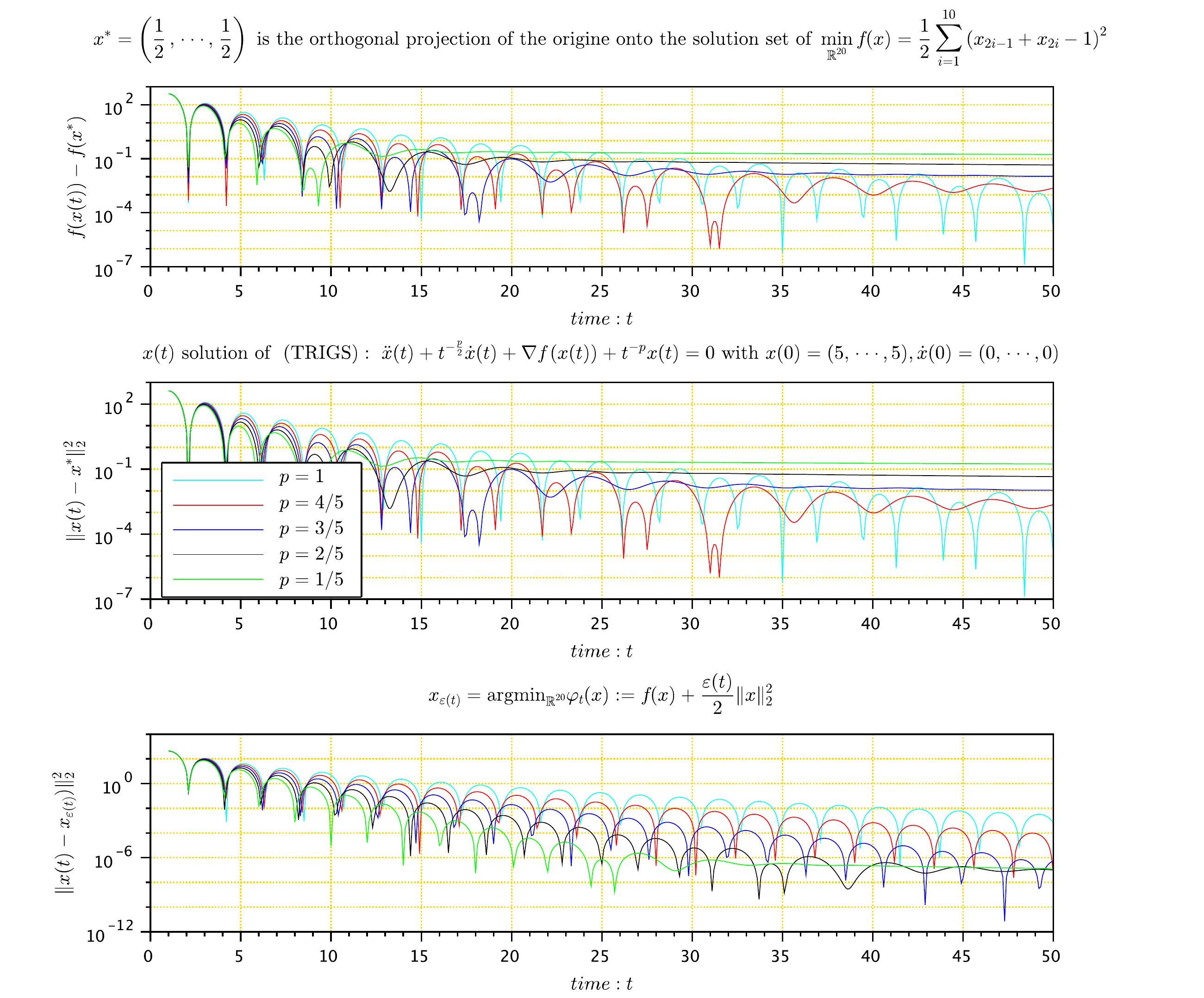}
\caption{ (TRIGS) with $f(x) := \frac12\sum_{i=1}^{10}\left(x_{2i-1}+x_{2i}-1\right)^2$. Convergence rates of
values $f(x(t))-f(x^*)$ (top), of $\|x(t)-x^*\|_2$ (middle), of $\|x(t)-x_{\varepsilon (t)}\|_2$ (below)} 
\label{fig:trials}
\end{figure}
The numerical experiments described in Figure \ref{fig:trials} are in  agreement with our
theoretical findings. 
They show the trade-off between the convergence rate of values $f(x(t))-f(x^*)$  and  trajectories $\|x(t)-x_{\varepsilon (t)})\|_2$, and that taken around $p=\frac{2}{3}$ is a good compromise. 
We limit our illustration to the case $0<p<1$. It is the most interesting case for obtaining  good convergence rate of the trajectories towards the minimum norm solution.  We also notice that  the regularization Tikhonov's term $\frac{1}{t^p} x(t)$ in the system {\rm(TRIGS)} reduces the oscillations.
This  suggests  introducing  the Hessian driven damping into these dynamics to further  dampen oscillations, see \cite{ACFR}, \cite{APR}, \cite{BCL} and references therein.
This  is related to the notion of strong damping for PDE's.

\bigskip

\section{Conclusion, perspective}\label{sec:Conclusion}
In the general framework of convex optimization in Hilbert spaces, we have introduced a damped inertial dynamic which generates trajectories  rapidly converging towards the minimum norm solution. We obtained these results by removing restrictive assumptions concerning the convergence of trajectories, made in previous works. 
This seems to be the first time that these two properties have been obtained for the same inertial dynamic.
We have developed an in-depth mathematical Lyapunov analysis of the dynamic which is a valuable tool for the development of corresponding results for algorithms obtained by temporal discretization. Precisely, the corresponding algorithmic study must be the subject of a subsequent study. 
Many interesting questions such as the introduction of Hessian-driven damping to attenuate oscillations, and the study of the impact of error disturbances, merit further study. 
These results also adapt well   to  inverse problems for which strong convergence of trajectories, and obtaining a solution close to a desired state are key properties. 
It is likely that a parallel approach can be developed for  constrained optimization, in multiobjective optimization for the dynamical approach to Pareto optima, and within the framework of potential games. The Lyapunov analysis developed in this paper could  also be very  useful to study the asymptotic stabilization  of several classes of PDE's, for example nonlinear damped wave equations.

\appendix
\section{Auxiliary results}

\subsection{Existence and uniqueness for the Cauchy problem, energy estimates}
Let us  first  show that the Cauchy problem for (TRIGS) is well posed.
The proof  relies on classical arguments combining   the Cauchy-Lipschitz theorem with energy estimates. The following proof has been given in \cite{AL}. We reproduce it for the convenience of the reader, and give supplementary energy estimates.

\begin{theorem}\label{Cauchy-weel-posed}
Let us make the assumptions $( \mathcal{H}_0) $ on $f$ and $\varepsilon$.
Then, given $(x_0, v_0) \in \cH \times \cH$, there exists a unique global classical solution $x : [t_0, +\infty[ \to \mathcal{H}$ of the Cauchy problem
\begin{align}\label{DynSys-1}
\begin{cases}
\ddot{x}(t) + \d\sqrt{\e(t)} \dot{x}(t) + \nabla f\left(x(t) \right)+\epsilon (t) x(t)=0 \vspace{1mm}\\
x(t_0) = x_0, \,
\dot{x}(t_0) = v_0.
\end{cases}
\end{align}
In addition, the global energy function $t \mapsto W(t)$ is decreasing where
$$
W(t):= \demi \| \dot{x}(t)\|^2 + f(x(t)) + \demi \epsilon (t) \|x(t)\|^2 ,
$$
and we have the energy estimate
\begin{equation}\label{energy1}
\int_{t_0}^{+\infty} \sqrt{\e(t)} \| \dot{x}(t)\|^2 dt <+\infty.
\end{equation}
\end{theorem}
\begin{proof} 
Consider the Hamiltonian formulation of  (\ref{DynSys-1}), which gives the first order system
\begin{align}\label{DynSys-2}
\begin{cases}
\dot{x}(t) - y(t) =0 \vspace{1mm} \\
\dot{y}(t) + \d\sqrt{\e(t)} y(t) + \nabla f\left(x(t) \right)+\epsilon (t) x(t)=0 \vspace{1mm}\\
x(t_0) = x_0, \,
y(t_0) = v_0.
\end{cases}
\end{align}
According to the hypothesis $( \mathcal{H}_0) $, and by applying the Cauchy-Lipschitz theorem in the locally Lipschitz case, we obtain the existence and uniqueness of a local solution of the Cauchy problem \eqref{DynSys-2}.
Then, in order to pass from a local solution to a global solution, we use  energy estimates.  By taking the scalar product of (TRIGS) with $\dot{x}(t)$, we obtain
\begin{equation}\label{energy2}
\frac{d}{dt} \Big( \demi \| \dot{x}(t)\|^2 + f(x(t)) + \demi \epsilon (t) \|x(t)\|^2 )  \Big)
+ \d\sqrt{\e(t)} \| \dot{x}(t)\|^2 - \demi \dot{\epsilon} (t) \|x(t)\|^2 =0.
\end{equation}
According to $(H_3)$, the function $\epsilon(\cdot)$ is non-increasing. Therefore, the energy function $t \mapsto W(t)$ is decreasing where
$
W(t):= \demi \| \dot{x}(t)\|^2 + f(x(t)) + \demi \epsilon (t) \|x(t)\|^2 .
$
The end of the proof follows a standard argument. Take a maximal solution defined on an interval $[t_0, T[$. If $T$ is infinite, the
proof is over. Otherwise, if $T$ is finite, according to the above energy estimate, we have that $\| \dot{x}(t)\|$ remains bounded, just like $\| x(t)\|$ and $\| \ddot{x}(t)\|$ (use (TRIGS)). Therefore, the limit of $x(t)$ and  $\dot{x}(t)$ exists when $t \to T$. Applying the local existence result  at $T$ with the  initial conditions thus obtained gives a contradiction to the maximality of the solution.\\
Let us complete the proof with the energy estimates. Returning to \eqref{energy2}, we get
\begin{equation}\label{energy3}
\frac{d}{dt} \Big( \demi \| \dot{x}(t)\|^2 + f(x(t)) + \demi \epsilon (t) \|x(t)\|^2 )  \Big)
+ \d\sqrt{\e(t)} \| \dot{x}(t)\|^2  \leq 0.
\end{equation}
After integration of \eqref{energy3}, we get \eqref{energy1}.
\end{proof}

\if
{
Et enfin, il faut peut être faire la remarque que l'hypothèse $H_1$ implique la condition $(CD)_\lambda$.
}
\fi


\begin{thebibliography}{10}


 \bibitem{AA} {\sc F. Alvarez, H. Attouch}, Convergence and asymptotic stabilization for some damped hyperbolic equations with non-isolated equilibria,
ESAIM Control Optim. Calc. Var.  6 (2001),  539--552.



\bibitem{AlvCab} {\sc F. Alvarez, A. Cabot},  Asymptotic selection of viscosity
equilibria of semilinear evolution equations by the introduction of a slowly vanishing term,  Discrete Contin. Dyn. Syst. 15 (2006),  921--938.


\bibitem{AAD1}{\sc V. Apidopoulos, J.-F. Aujol,  Ch. Dossal},
 The differential inclusion modeling the FISTA algorithm and optimality of convergence rate in the case $b \leq 3$, SIAM J. Optim.,  28(1)  (2018),  551---574.

\bibitem{Att-book} {\sc H. Attouch},  Variational convergence for functions and operators, Applicable Mathematics
   Series, Pitman Advanced Publishing Program, 1984.

\bibitem{Att2} {\sc H. Attouch},  Viscosity solutions of minimization problems,
SIAM J. Optim. 6 (3) (1996), 769--806.



\bibitem{ABotCest}  {\sc H. Attouch, R.I. Bo\c t,  E.R. Csetnek},
Fast optimization via  inertial dynamics  with closed-loop damping,
Journal of the European Mathematical Society (JEMS), 2021, hal-02910307.

\bibitem{abc} {\sc H. Attouch, L.M. Brice\~no-Arias, P.L. Combettes},   A parallel splitting method for coupled monotone inclusions,
 SIAM J. Control Optim. 48 (5) (2010), 3246--3270.


\bibitem{abc2} {\sc H. Attouch, L.M. Brice\~no-Arias, P.L. Combettes},   A strongly convergent primal-dual method for nonoverlapping domain decomposition,
 Numerische Mathematik, 133(3)  (2016),  443--470.



\bibitem{AC10} {\sc H. Attouch, A.  Cabot},  Asymptotic stabilization of inertial gradient dynamics with time-dependent viscosity,   J. Differential Equations, 263 (9), (2017), 5412--5458.



\bibitem{ACFR} {\sc  H. Attouch, Z. Chbani, J. Fadili, H. Riahi}, First order optimization algorithms via inertial  systems with Hessian driven damping,  Math. Program. (2020),  https://doi.org/10.1007/s10107-020-01591-1.


\bibitem{ACPR}
{\sc H. Attouch, Z. Chbani, J. Peypouquet, P. Redont}, Fast convergence of inertial dynamics and algorithms with asymptotic vanishing viscosity, Mathematical Programming, 168 (1-2) (2018),  123--175.


\bibitem{ACR} {\sc H. Attouch,  Z. Chbani, H. Riahi},
 Combining fast inertial dynamics for convex optimization with Tikhonov regularization,
J. Math. Anal. Appl, 457 (2018),  1065--1094.


\bibitem{AttCom} {\sc H. Attouch,  R. Cominetti},  A dynamical approach to convex
minimization coupling approximation with the steepest descent method,   J.
Differential Equations, 128 (2) (1996), 519--540.



\bibitem{AttCza1}{\sc H. Attouch, M.-O. Czarnecki},  Asymptotic control and stabilization
of nonlinear oscillators with non-isolated equilibria, J. Differential Equations 179 (2002), 278--310.


\bibitem{AttCza2} {\sc H. Attouch, M.-O. Czarnecki},  Asymptotic behavior of coupled dynamical systems with multiscale aspects, J. Differential Equations 248 (2010), 1315--1344.


\bibitem{AttCzaPey1} {\sc H. Attouch, M.-O. Czarnecki, J. Peypouquet},  Prox-penalization and splitting methods for constrained variational problems, SIAM J. Optim. 21 (2011), 149--173.


\bibitem{AttCzaPey2} {\sc H. Attouch, M.-O. Czarnecki, J. Peypouquet},   Coupling forward-backward with penalty schemes and parallel splitting for constrained variational inequalities,  SIAM J. Optim. 21 (2011), 1251--1274.


\bibitem{Att-Czar-last} {\sc H. Attouch, M.-O. Czarnecki},  Asymptotic behavior of gradient-like dynamical systems involving inertia and multiscale aspects, J. Differential Equations,  262 (3)  (2017),  2745--2770.

\bibitem{AL} {\sc H. Attouch, S. L\'aszl\'o},  Convex optimization via inertial algorithms with vanishing Tikhonov  regularization: fast convergence to the minimum norm solution,
arXiv:2104.11987v1 [math.OC] 24 Apr 2021.


\bibitem{AP} {\sc H. Attouch,  J. Peypouquet},  The rate of convergence of Nesterov's accelerated  forward-backward method is actually faster than $1/k^2$, SIAM J. Optim., 26(3) (2016),  pp. 1824--1834.


\bibitem{APR} {\sc H. Attouch, J. Peypouquet, P. Redont},   Fast convex minimization via inertial dynamics with Hessian driven damping,   J. Differential Equations, 261(10),  (2016),  5734--5783.


\bibitem{BaiCom} {\sc J.-B. Baillon, R. Cominetti},  A convergence result for non-autonomous subgradient evolution equations and its application to the steepest descent exponential penalty trajectory in linear programming,  J. Funct. Anal. 187 (2001) 263-273.



\bibitem{BC}{\sc H. Bauschke, P. L. Combettes},   Convex Analysis and Monotone Operator Theory in Hilbert spaces, CMS Books in Mathematics, Springer,   (2011).






\bibitem{BotCse1} {\sc R. I. Bot, E. R. Csetnek},  Forward-Backward and Tseng's type penalty schemes for monotone inclusion problems,  Set-Valued Var. Anal. 22 (2014), 313--331.

\bibitem{BCL}
{\sc R. I. Bo\c t, E. R. Csetnek, S.C. L\'aszl\'o},     Tikhonov regularization of a second order dynamical system with Hessian damping, Math. Program. (2020), https://doi.org/10.1007/s10107-020-01528-8.

\bibitem{Bre1}{\sc H. Br\'ezis},   Op\'erateurs maximaux monotones dans les espaces de Hilbert et \'equations d'\'evolution, Lecture Notes 5, North Holland, (1972).

 \bibitem{Cabot-inertiel}{\sc A. Cabot},  Inertial gradient-like dynamical system controlled by a stabilizing term, J. Optim. Theory Appl. 120 (2004) 275--303.

\bibitem{Cab} {\sc A. Cabot},  Proximal point algorithm controlled by a slowly vanishing term: Applications to hierarchical minimization, SIAM J. Optim.  15 (2) (2005), 555--572.


 \bibitem{CEG}{\sc  A. Cabot, H. Engler, S. Gadat},  On the long time behavior of second order differential equations
with asymptotically small dissipation,
Trans. Amer. Math. Soc.  361 (2009), 5983--6017.



\bibitem{CD}{\sc  A. Chambolle, Ch. Dossal},  On the convergence of the iterates of Fista,
J. Opt. Theory Appl., 166 (2015), 968--982.




\bibitem{Com} {\sc R. Cominetti},   Coupling the proximal point algorithm with approximation methods, J. Optim. Theory Appl. 95 (3) (1997), 581--600.



\bibitem{CPS} {\sc R. Cominetti, J. Peypouquet, S. Sorin},  Strong asymptotic convergence of evolution equations governed by maximal monotone operators with Tikhonov regularization,  J. Differential Equations, 245 (2008),  3753--3763.



\bibitem{FM} {\sc  A. Fiacco,   G. McCormick},  Nonlinear programming: Sequential
Unconstrained Minimization Techniques, John Wiley and Sons, New York, (1968).


\bibitem{HJ2} {\sc A. Haraux, M.A. Jendoubi},
 A Liapunov function approach to the stabilization of second-order coupled systems, (2016) arXiv preprint arXiv:1604.06547.



\bibitem{Hirstoaga}{\sc S.A. Hirstoaga},
 Approximation et r\'esolution de probl\`emes d'\'equilibre, de point fixe et
d'inclusion monotone.  PhD thesis, Universit\'e Pierre et Marie Curie - Paris VI, 2006, HAL Id: tel-00137228.


\bibitem{JM-Tikh}{\sc M.A. Jendoubi, R. May},  On an asymptotically autonomous system with Tikhonov type
regularizing term, Archiv der Mathematik 95 (4) (2010),  389--399.




\bibitem{Nest1}{\sc  Y. Nesterov},   A method of solving a convex programming problem with convergence rate
$O(1/k^2)$, Soviet Math. Dokl.  27 (1983), 372--376.



\bibitem{Nest2}{\sc  Y. Nesterov},  Introductory lectures on convex optimization: A basic course, volume 87 of
Applied Optimization. Kluwer Academic Publishers, Boston, MA, 2004.








\bibitem{Polyak}{\sc  B. Polyak},     Introduction to Optimization, New York, NY: Optimization Software-Inc,  1987.


\bibitem{Siegel} {\sc W. Siegel},  Accelerated first-order methods: Differential equations and Lyapunov functions, arXiv:1903.05671v1 [math.OC], 2019.



\bibitem{SBC}{\sc W.  Su,  S. Boyd,  E. J. Cand\`es}, 
A Differential Equation for Modeling Nesterov's
Accelerated Gradient Method: Theory and Insights.
NIPS, December 2014.



\bibitem{Tikh}{\sc A. N. Tikhonov}, Doklady Akademii Nauk SSSR 151 (1963) 501--504, (Translated in "Solution of incorrectly formulated problems and
the regularization method". Soviet Mathematics 4 (1963) 1035--1038).


\bibitem{TA}{\sc A. N. Tikhonov,  V. Y. Arsenin},  Solutions of Ill-Posed Problems,
Winston, New York, 1977.



\bibitem{Torralba}{\sc D. Torralba},  Convergence epigraphique et changements d'\'echelle en analyse variationnelle et optimisation,
PhD thesis, Universit\'e Montpellier, 1996.


\bibitem{WRJ} {\sc A. C. Wilson, B. Recht, M. I. Jordan}, 
A Lyapunov analysis of momentum methods in
optimization, arXiv preprint arXiv:1611.02635, 2016.


\end{thebibliography}
\end{document}